\documentclass[10pt]{article}
\usepackage{enumerate,indentfirst}
\usepackage{fontenc,textcomp,latexsym}
\usepackage{amsfonts,amsmath,amsthm,amssymb}
\usepackage[english]{babel}
\usepackage[latin1]{inputenc}
\usepackage{hyperref}
\usepackage{color}
\def \tbar{\tilde}
\def \rw {\rightarrow}
\def \be{\begin{equation}}
\def \cadlag {\mx{c\`adl\`ag}}
\def \ee{\end{equation}}

\def \ss {\mathcal{S}_{c}^{2}}

  \def \cft {(\cf_t)_{t\ge 0}}
  \def \rl {\R^\ell}
  \def \d{\Delta}
  \def \nn {\nonumber}
  \def \barl{\begin{array}{l}}
\def \ear{\end{array}}
\newtheorem{Assumption}{Assumption}[part]

\newtheorem{Corollary}{Corollary}[part]

\newtheorem{Lemma}{Lemma}[part]
\newtheorem{Proposition}{Proposition}[part]
\newtheorem{Remark}{Remark}[part]
\newtheorem{Theorem}{Theorem}[part]
\newcommand{\biindice}[3]%
{

\begin{array}[t]{c}
#1\\
{\scriptstyle #2}\\
{\scriptstyle #3}
\end{array}
}
\def \frto{\forall t\ge 0}

\newtheorem{\theHypothes}{\thesection.\arabic{Hypothes}}

\newcommand{\nc}{\newcommand}
\def \esssup{\mathop{\mathrm{ess\,sup}}}
\nc{\essinf}{\mathop{\mathrm{ess\,inf}}}
\nc{\argmax}{\mathop{\mathrm{arg\,max}}}
\def \qq {\qquad}

\def \ms{\medskip}

\def \P{\mathbb{P}}
\def \N{\mathbb{N}}
\def \R{\mathbb{R}}

\def \E{\mathbb{E}}
\def \F{\mathcal{F}}

\def \1{{\bf 1}}
\def \gs{\overline{\gamma}}
\def \A{\mathcal{A}}

\def \Fc{{\cal F}}

\def \Hc{{\cal H}}

\def \lb{\label}

\def \xa{\1_{[\xi=a]}}
\def \esssup{{\mbox{esssup}}}

\def \ts{\bar{\tau}}
\def \bt{{\bar{\tau}'}}
\def \bb{{\bar{\beta}'}}

\def\beqs{\begin{eqnarray*}}
\def\enqs{\end{eqnarray*}}
\def\beq{\begin{eqnarray}}
\def\enq{\end{eqnarray}}

\newcommand{\Sum} {\displaystyle \sum\limits}                

\newcommand{\Int} {\displaystyle \int}                       
\newcommand{\Sup}     {\displaystyle \sup\limits}

\def \t {\tau}
\def \b{\beta}
\def \cp {\mathcal{P}}
\title{Optimal stochastic impulse control
problem with delay with actions decided at the execution time}
\date{\today}
\author{Said Hamadene
\thanks{LMM, Le Mans University. e-mail:
\href{mailto: said.hamadene@univ-lemans.fr}{said.hamadene@univ-lemans.fr}. The author research is part of the ANR project DREAMeS (ANR-21-CE46-0002).}~~~ Ibtissam Hdhiri
\thanks{Faculty of sciences of Gab\`{e}s \& LR17ES11, University of Gab\`es. e-mail: \href{mailto: Ibtissem.Hdhiri@fsg.rnu.tn}
{: Ibtissem.Hdhiri@fsg.rnu.tn}} \\}
\vspace{2.cm}

\begin{document}
\maketitle

\begin{abstract}
In this paper, we consider a class of stochastic impulse control problem when there is  a fixed delay $\Delta$ between the decision and execution times. The dynamics of the controlled system between two impulses is an arbitrary adapted stochastic process. Unlike the most existing literature, we consider the problem when the impulse sizes are decided at the execution time in both risk-neutral and risk-sensitive cases. This model fits more, in the real life, for some problems such as the pricing of swing options. The horizon $T$ of the problem can be finite or infinite. In each case we show the existence of an optimal impulse strategy. The main tools we use are the notions of reflected Backward Stochastic Differential Equations (BSDEs for short) and the Snell envelope of processes.
\end{abstract}
\vskip 0.2cm AMS subject Classifications: 60G40; 60H10; 93E20.
\vskip 0.2cm\noindent {\bf Key words~:}
Impulse control; Execution delay; Snell envelope; Stochastic control; Backward stochastic
differential equations; Optimal stopping time; Swing options.
\def \dl {\delta}

\section{Introduction}
The impulse control  problem consists of acting on a given system whose dynamics is described by a stochastic process ${(X_t)}_{t \in \mathbb T}$, at discrete times $\tau_k$ with impulses $\beta_k$, $k\ge 1$ ($\t_k\le \t_{k+1}$). The actions are not free and then need to be implemented in such a way to achieve a certain level of profitability, or even optimal profitability. The sequences $(\t_k)_{k\ge 1}$ and $(\beta_k)_{k\ge 1}$ are the parameters of interventions of the controller on the system and 
$\dl:=(\t_k,\b_k)_{k\ge 1}$ is called a strategy of impulse control. For $k\ge 1$, let $\psi(\beta_k)$ be the cost for the controller when she/he intervenes on the system at time $\t_k$ with the action $\beta_k$. The expected total payoff is then given by: 
\be\lb{costintro}
\E\big[\Phi\big( \int_{[0,T)} g(s,X_s^{\delta}) ds - \sum _{k \geq 1} \psi(\xi_k) \1_{[\tau_k < T]} \big)\big],
\ee
where $X^{\delta}$ stands for the controlled dynamics when the strategy $\dl$ is implemented, $g$ is the instantaneous reward function, $\Phi$ the utility function of the controller. The constant $T$ is finite (resp. infinite) when the horizon is finite (resp. infinite). The objective is then to find an optimal strategy, i.e., a strategy $\dl^*$ which maximizes \eqref{costintro}. 

The case $\Phi(x)=x$ (resp. $\Phi(x)=e^{\theta x}$) is called risk-neutral (resp. risk-sensitive). 

In the case when $X^\dl$ is Markovian in between two successive impulses (i.e. on the time interval $[\t_k,\t_{k+1}[$ for any $k\ge 0$ ($\t_0=0$)), the stochastic impulse control problem is solved mainly by using deterministic techniques (one can see e.g. \cite{Bensoussan, Korn, Harison, Jeanblanc, Oksendal2,lep-sm,Keppo}, etc. the list is far from being exhaustive). Now, when the dynamics is not necessarily Markovian, the impulse control problem was considered for the first time by Djehiche et al. \cite{Djehiche} who proposed a new approach based on the notions of Backward Stochastic Differential Equations and Snell envelope of processes. Later on, several works have been carried out on the subject including \cite{hdhirikrouf1,   perninge,  agram, hdhirizaatra} to quote a few. 

When the execution of impulses requires a preparatory work, a time delay $\Delta$ is integrated into the model. The decisions are taken at $\tau_k$ and executed at $\tau_k + \Delta$, see \cite{hdhirikrouf2,bhhz,Bayraktar,Bruder, Oksendal2}, etc. for more details. 
Except in \cite{Bayraktar}, in those works, the amplitude of the impulse $\b_k$ is decided at the decision time $\t_k$. However, in several realistic situations, although the controller takes decision at $\tau_k$, she/he may choose or be required to wait until the execution time $\t_k+\d$ before to determine the size of the intervention. 

This feature appears in the pricing of swing options which are options in markets of commodities (see e.g. \cite{pfbr,jrt,ew,diner,yeo} and the references therein). They consist of a contract which delineates the least and most energy that the option holder can buy (or "take") per day and per month, how much that energy will cost (known as its strike price), and how many times $N_c$ during the month she/he can change or "swing" the daily quantity of energy purchased.

Swing options are most commonly used for the purchase of oil, natural gas, and electricity. They may be used as hedging instruments by the option holder, to protect against price changes in these commodities. For example, a power company might use a swing option to manage changes in customer demand for electricity that occur throughout the month as temperatures rise and fall. These contracts are more intricate than they appear to be. Consequently, they tend to make valuation challenging. An oil company might also do the same with fuel for customer demand for heat during winter months.

The price paid for the commodity can either be fixed or floating. A floating or "indexed" price essentially means that it is linked to the price in the market. In contrast to a fixed price contract, an indexed price contract is less flexible.

There are several works on swing options including \cite{diner,carday,ew,jrt,pfbr,yeo} to quote a few. This outstanding feature
of the embedded options is reminiscent of American contingent claims with multiple exercises. This point of view has been well documented in the paper by Carmona-Touzi \cite{carto}, in the case when maturity is finite or infinite, aiming mainly to find optimal times of interventions 
$(\t^*_i)_{i=1,N_c}$. In \cite{carto}, the authors assume a refraction period which is a condition which looks like to our delay, i.e. $\t^*_i-\t^*_{i-1}\ge \Delta$ 
in order to 
prevent $(\t^*_i)_{i=1,N_c}$ from bunching
up. On the other hand, as mentionned previously, the authors assume that the dynamics of the price of the commodity is not affected by the exercises. Namely we are in the framework of a small investor. However one can think of that the option integrates prices by layers, in which case the effective price depends on the amount of the commend.  For example, the option can stipulate that the price increases with the amount of the commend. If moreover there is a refraction period and the amount of the commend is made when it is executed, then pricing this option falls exactly in the model we study in this paper. 

The novelty of this paper is to consider the impulse control problem for general stochastic processes in the case when there is a fixed delay $\d$ in the interventions and the amplitudes of the impulses are decided at the moment when those latter occur and not when the decisions to impulse are made. We study the cases of finite and infinite horizons and risk-neutral and risk-sensitive utility functions.  Note that, in \cite{Bayraktar} the authors deal with the so-called Markovian setting, by using deterministic methods, since between two impulses the  dynamics of the controlled system is solution of a standard Markovian stochastic differential equation. On the other hand, the model we consider differs from the one treated in \cite{bhhz} because in this latter the sizes $\xi_n$ of the impulses are decided at $\t_n$, when the decision to make an  impulse is made. 

The rest of the paper is organized as follows. 
In section 2, we consider the problem of stochastic impulse control with finite horizon $T$ and fixed delay $\d$ when the sizes of the impulses are fixed at   the execution times,  in both the risk-neutral and risk-sensitive  settings.   The delay $\d$ makes that it is only possible to implement  
$\mathfrak{X}=[\frac{T}{\Delta}]$ impulses at most before the horizon $T$. In both cases we prove the existence of an optimal strategy which we exhibit. Mainly we use reflected BSDEs (resp. Snell envelope of processes) to deal with the risk-neutral (resp. risk-sensitive) framework.  In Section 3, we consider the infinite horizon setting of the problem in the risk-neutral case. Once more, by using reflected BSDEs in infinite horizon, we show the existence of an optimal strategy which we exhibit. The risk-sensitive case can be deduced in a similar way (see Remark \ref{rmkrs}). Finally, at the end of the paper we gather, as an Appendix, some results, on the projections of stochastic processes and on the Snell envelope of processes, to make as much as possible the paper self-contained.  

\setcounter{equation}{0} \setcounter{Assumption}{0}
\setcounter{Theorem}{0} \setcounter{Proposition}{0}
\setcounter{Corollary}{0} \setcounter{Lemma}{0}
\setcounter{Definition}{0} \setcounter{Remark}{0}
\section{The finite horizon setting}
\subsection{Preliminaries and description of the model}
\setcounter{equation}{0} \setcounter{Assumption}{0}
\setcounter{Theorem}{0} \setcounter{Proposition}{0}
\setcounter{Corollary}{0} \setcounter{Lemma}{0}
\setcounter{Definition}{0} \setcounter{Remark}{0}

Let $T>0$ be a real number which is the horizon of the problem. Let $(\Omega,\mathcal{F},\P)$ be a complete probability space on which is defined a standard $d$-dimensional
Brownian motion $B=(B_{t})_{t\leq T}$. Let $(\Fc_{t}^{0}:= \sigma\{B_{s}, s\leq t\})_{t\leq T}$ be the natural filtration of $B$ and $(\mathcal{F}_{t})_{t\leq T}$ its completion with the $\P$-null sets of $\Fc$, therefore $(\mathcal{F}_{t})_{t\leq T}$ satisfies the usual conditions. 
 The Euclidean norm of an element $x\in\R^{n}$ is denoted by $|x|$. We also need the following notations ($p>1$ and $n\ge 1$):  
\begin{enumerate}
\item [$\bullet$] $\mathcal{P}$ is the $\sigma$-algebra on $[0,T]\times\Omega$
 of $\mathcal{F}_{t}$-progressively measurable processes;
\item [$\bullet$] $\mathcal{H}^{p,n}=\{(v_{t})_{t\leq T}:$ $\mathcal{P}$-measurable, $\R^{n}$-valued process s.t.
$\E[\int_{0}^{T}|v_{s}|^{p}\,ds]<\infty\}$;
\item [$\bullet$] $\mathcal{S}^{2}=\{Y,$ $\mathcal{P}$-measurable  c\`{a}dl\`{a}g (or rcll) process, such that $\E[\sup_{0\leq t\leq T}|Y_{t}|^{2}]<\infty\}$;

\item [$\bullet$] $\mathcal{S}_{c}^{2}$ (resp. 
$\mathcal{S}_{i}^{2}$) is the subset of 
$\mathcal{S}^{2}$ of continuous processes (resp. continuous non-decreasing 
processes $(k_{t})_{t\leq T}$ such that $k_0=0$);

\item [$\bullet$]$\mathcal{T}_{t_0}=\{\nu: (\mathcal{F}_{t})_{t\leq T}$-stopping time, such that $\P$-a.s, $t_0\leq \nu \leq
T\}$ $(t_0\le T)$.
\end{enumerate}
\label{sec3}
\setcounter{equation}{0} \setcounter{Assumption}{0}
\setcounter{Theorem}{0} \setcounter{Proposition}{0}
\setcounter{Corollary}{0} \setcounter{Lemma}{0}
\setcounter{Definition}{0} \setcounter{Remark}{0}
\def \ch {\mathcal{H}}
\def \d{\Delta}
We consider a system, whose dynamics is described by a stochastic process $(X_t)_{t\le T}$ of $ \ch^{2,\ell}$ ($\ell \ge 1$), which is subject to impulses which are implemented with some fixed delay $\d<T$. 
Precisely when a decision to impulse the system at $\tau$ is made, it will be executed with a fixed lag $\Delta$, i.e., at time $\tau+\d$. On the other hand, 
unlike the most of existing works, we assume that the size of the impulse is not determined at the   time of decision $\tau$ but at the execution one $\tau+\d$.  Consequently an impulse strategy $\delta:=(\tau_{n},\xi_{n})_{n\geq 1}$ of the controller consists of a pair of sequences $(\tau_{n})_{n\geq 1} \in \mathcal T_0$ and 
$(\xi_{n})_{n\geq 1}$ such that for any $n\ge 1$: 
\medskip
\def \nd{\noindent}
\def \ms{\medskip}

\noindent (a) $\min\{T,\t_n+\d\}\le \t_{n+1}\le T$.
\medskip
 
\noindent (b) $\,\xi_{n}$ are $\mathcal{F}_{(\tau_{n}+\Delta)\wedge T}$-measurable random variables taking values in  a finite subset \\$U=\{a_1,\ldots,a_p\}$ of $\R^{\ell}$. 
\medskip

\noindent The strategy $\delta=(\tau_{n},\xi_{n})_{n\geq 1}$ is called admissible  if it satisfies the conditions (a)-(b).
  \medskip
 
In this definition of admissible strategies, point (a) means that as far as the horizon of the problem is not reached, there is a delay $\d$ between two impulses while point (b) tells us that the size of the impulse is decided when it occurs. 
\vskip0.2cm

\begin{Remark}\lb{remadmstrat}The condition (a) above implies also that for any $n\ge 1$, $\t_n\le \t_{n+1}$. Actually it is enough to consider $n_0=\inf\{n\ge 0, \t_n+\d>T\}$ and to show the property for $n\le n_0$ and for $n\ge n_0+1$. \qed
\end{Remark}

The set of
admissible strategies will be denoted by $\mathcal{A}$ and for $m\geq 0$, $\A_m$ will be the set of admissible strategies $\delta = {(\tau_n,\xi_n)}_{n \geq 1}$ such that  for all $ n> m$,  $\tau_n = T$, i.e., they are strategies which allow for $m$ impulses at most.  For every $\delta \in \A_m$ or $\A$, when mentioned, we assume that $\t_0=0$ and { $\xi_{0}=0$}. Finally let us point out that in our model the values of $\xi_n$ are irrelevant when  $\t_n \ge T-\Delta$ (see the definition of the payoff in \eqref{cost}).
\medskip
\def \mx{\mbox}

If an admissible strategy or control policy $\delta=(\tau_{n},\xi_{n})_{n\geq 1}$ is implemented, then the dynamics $X^{\delta}$ of the controlled process is given by: 
\be \lb{xd}\forall t<T,\,\, X_{t}^{\delta} =X_{t}+\sum_{n\geq 1}\xi_{n}\1_{[\tau_{n}+\Delta \leq t]} \mx{ and } X_{T}^{\delta}=\lim_{s\nearrow T}X_{s}^{\delta}.\ee
This definition of the process $X^{\delta}$ means that the controller does not make an impulse at the terminal time $T$. 
The associated total expected reward  for the controller is given by:
\begin{equation}\lb{eqpayoff}
J(\delta):=\E[\Phi\{\int_{0}^{T}g(s,X_{s}^{\delta})\;ds-\sum_{n\geq
1}\psi(\xi_{n})\1_{[\tau_{n}< T-\Delta]}\}], 
\end{equation}
where $g(.)$ stands for the instantaneous reward function and $\psi(a)$ is the cost paid by the agent when exerting a control of magnitude $a$ which we assume non-negative. The function $\Phi$ is the utility function of the controller. In our frameworks, it will be first chosen as the identity (resp. exponential) function in Subsection \ref{neutral} (resp. \ref{risk}).

This work aims to find the optimal admissible strategy $\delta^{\ast}={(\tau_{n}^{\ast},\xi_{n}^{\ast})}_{n\geq 1}$ of the controller, i.e., which maximizes the  expected reward  (\ref{eqpayoff}) and then satisfies 
$$J(\delta^{\ast})=\Sup_{\delta\in \mathcal{A}}J(\delta).$$

Since every two consecutive interventions are separated by at least a fixed lag $\Delta$, then we have:
\begin{Proposition}\label{sup}
Let $\mathfrak{X}=[\frac{T}{\Delta}]$. We then have:
  \be \label{ega2}\Sup_{\delta\in \mathcal{A}}J(\delta)=\Sup_{\delta\in \mathcal{A}_{\mathfrak{X}}}J(\delta).\ee
\end{Proposition}
\def \ca {\mathcal{A}}
\begin{proof}Actually $\Sup_{\delta\in \mathcal{A}}J(\delta)\ge \Sup_{\delta\in \mathcal{A}_\mathfrak{X}}J(\delta)$ since $\mathcal{A}_{\mathfrak{X}}\subset \mathcal{A}$. On the other hand for any $\delta \in \ca$, $\delta$ belongs obviously to $\ca_\mathfrak{X}$, which implies the reverse inequality and then the equality between the two hand-sides of \eqref{ega2}. \end{proof}
\medskip

Next on the functions $g$ and $\psi$ we assume: 
\noindent \begin{Assumption}\label{assumpt2}
${}$\\
i) The  mapping $(t,\omega,x)\in [0,T]\times\Omega\times  \R^{\ell} \mapsto g(t,\omega,x)\in \R$ is $\mathcal{P}\otimes \mathcal{B}(\R^{\ell})/\mathcal{B}(\R)$-measurable. 
Moreover there exists a non-negative process $(\gamma_t)_{t\le T}$ of $\mathcal{H}^{2,1}$ such that:$$\P-a.s. \mbox{ for any }(t,x)\in[0,T]\times\R^{\ell},\,|g(t,\omega,x)|\le \gamma_t(\omega).$$

\noindent ii) For any $a\in \R^{\ell}$, $\psi(a)\ge 0
.$\end{Assumption}
\subsection{The risk neutral impulse control problem} \label{neutral}
In this part, we consider the standard  risk neutral case of our problem i.e. $\Phi(x)=x$ in \eqref{eqpayoff}. The total expected reward associated with an admissible strategy $\delta=(\tau_{n},\xi_{n})_{n\geq1}$ is then given by:
\begin{eqnarray} \label{cost}
J(\delta)&:=&\E[\int_{0}^{T}g(s,X_{s}^{\delta})\;ds-\sum_{n\geq
1}\psi(\xi_{n})\1_{[\tau_{n}<T-\Delta]}]\\
&=& \E[\int_{0}^{(\tau_1+\d)\wedge T}g(s,X_{s})\;ds+ \sum_{n\geq
1} \{\int_{{(\tau_n+\d)\wedge T}}^{{(\tau_{n+1}+\d)\wedge T}}g(s,X_{s}+\xi_1+\ldots +\xi_{n})\;ds-\psi(\xi_{n})\1_{[\tau_{n}<T-\Delta]}\}] \nonumber.
\end{eqnarray}
We are going to show that an optimal admissible strategy exists which is exhibited indeed. 
\subsubsection{Construction of the optimal strategy. An iterative scheme via reflected BSDEs.}
\label{sec4}

 For any $a \in \R^{\ell}$, let $Y^0(a)$ be the process defined by: 
\def \lb{\label}
\begin{equation}\label{y0ec} \forall t\le T, \,\,
Y_t^0(a) = \E\big[\int_t^T g(s,X_s+a) ds|\F_t \big]=\E\big[\int_0^T g(s,X_s+a) ds|\F_t \big]-\int_0^t g(s,X_s+a) ds. \end{equation}
Note that since $g$ verifies Assumption \ref{assumpt2}-i), then by Doob's maximal inequality (see \cite{Revuz}, pp.54), the process $Y^0(a)$ belongs to $\ss$.

Next for any $n \geq 1$, let $(Y^{n}(a),K^{n}(a),Z^{n}(a))$ be the solution of the following reflected BSDE:
\def \mx{\mbox}
\begin{equation}
\label{oordn}
\left\{
  \begin{array}{ll}
    Y^{n}(a)\in \mathcal{S}_c^{2}, \,Z^{n}(a)\in \mathcal{H}^{2,d} \mx{ and } K^{n}(a)\in \mathcal{S}_{i}^{2}; \\
    Y_{t}^{n}(a)=\int_{t}^{T}g(s,X_{s}+a)     \,ds+K_{T}^{n}(a)-K_{t}^{n}(a)-\int_{t}^{T}Z_{s}^{n}(a)\,dB_{s},\,\,t\leq T;\\
      Y_{t}^{n}(a)  \geq  O_{t}^{n}(a),\,\forall t\le T \mbox{ and } \Int_{0}^{T}(Y_{t}^{n}(a)-O_{t}^{n}(a))\,dK_{t}^{n}(a)=0.
  \end{array}
\right.
\end{equation}
The barrier  $(O_{t}^{n}(a))_{t\le T}$, which depends on $Y^{n-1}(a)$, is given by: $\forall t\le T$,
 \begin{align*}
      O_{t}^{n}(a):&=\E\big[\Int_{t}^{(t+\Delta)\wedge T}g(s,X_{s}+a)\,ds|\mathcal{F}_{t}\big]
         +\1_{[t< T-\Delta]} \E[\max_{\beta\in U} \{-\psi(\beta)+Y_{t+\Delta}^{n-1}(a+\beta)\}|\mathcal{F}_{t}]\\
         &=\E[\underbrace{\int_{t}^{(t+\Delta)\wedge T}g(s,X_{s}+a)\,ds         +\1_{[t< T-\Delta]} \max_{\beta\in U} \{-\psi(\beta)+Y_{t+\Delta}^{n-1}(a+\beta)\}}_{{\cal L}^n_t(a)}|\mathcal{F}_{t}]
\end{align*}
 Let us notice that  the process $(O_{t}^{n}(a))_{t\le T}$ is the predictable projection of the measurable process 
$({\cal L}^n_t(a))_{t\le T}$ (see Appendix \ref{section-projection}). 
\def \ft{\forall t\le T, \,\,}
According to Theorem 1.3 in \cite{Hamadene1}, if 
$(O_{t}^{n}(a))_{t\le T}$ is rcll with positive jumps only and uniformly square integrable then the solution of the BSDE \eqref{oordn} exists and is unique and the first component $Y^n(a)$ of the solution has the following representation:  
\begin{equation}
\ft Y_{t}^{n}(a)=\esssup_{\tau \in
\mathcal{T}_t}\E[\int_{t}^{\tau}g(s,X_{s}+a)\,ds+  O_{\tau}^{n}(a)|\mathcal{F}_{t}].
\label{ordrenn}\end{equation}

\def \cf {\mathcal{F}}

\begin{Proposition} \lb{yksi} For any $n\ge 1$ and $a\in \rl$, the triple $(Y^{n}(a),Z^n({a}),K^{n}(a))$ is well-posed. 
\label{pro limit}
\end{Proposition}

\begin{proof} 
We proceed by induction. First note that, as mentionned previously, for any $a\in \rl$, the process $Y^0(a)$ belongs to $\ss$. Now recall that for any $t\le T$,
 \begin{eqnarray*}
      O_{t}^{1}(a):=\E\big[\Int_{t}^{(t+\Delta)\wedge T}g(s,X_{s}+a)\,ds|\mathcal{F}_{t}\big]
         +\1_{[t< T-\Delta]} \E[\max_{\beta\in U} \{-\psi(\beta)+Y_{t+\Delta}^{0}(a+\beta)\}|\mathcal{F}_{t}].
\end{eqnarray*}
Therefore this process is obviously uniformly $d\,\P$-square integrable. Next the processes\\
$(\E\big[\Int_{t}^{(t+\Delta)\wedge T}g(s,X_{s}+a)\,ds|\mathcal{F}_{t}\big])_{t\le T}$ and $
       (\E[\max_{\beta\in U} \{-\psi(\beta)+Y_{t+\Delta}^{0}(a+\beta)\}|\mathcal{F}_{t}])_{t\le T}$ are continuous as they are the predictable projections of continuous processes respectively since the filtration $(\cf_t)_{t\le T}$ is Brownian. Consequently the process 
       $(O_{t}^{1}(a))_{t\le T}$ is uniformly $d\P$-square integrable and continuous on $[0,T]-\{T-\d\}.$ So let us analyze the discontinuity of 
       $(O_{t}^{1}(a))_{t\le T}$ in $T-\d$. Actually 
$$
\lim_{t\searrow T-\d }O_{t}^{1}(a)=\E\big[\Int_{T-\d}^{ T}g(s,X_{s}+a)\,ds|\mathcal{F}_{T-\d}\big]$$
and 
$$
\lim_{t\nearrow T-\d }O_{t}^{1}(a)=\E\big[\Int_{T-\d}^{ T}g(s,X_{s}+a)\,ds|\mathcal{F}_{T-\d}\big]+\max_{\beta\in U} (-\psi(\beta))$$
since $Y_{T}^{0}(a+\beta)=0$ and the filtration $(\cf_t)_{t\le T}$ is Brownian  (see e.g \cite{Dellacherie} VI, Theorem 90). As the function $\psi\ge 0$ then the jump of
$(O_{t}^{1}(a))_{t\le T}$ at $T-\d$ is non negative which implies,   by Theorem 1.3 in \cite{Hamadene1}, that the processes $(Y^1(a),Z^1(a),K^1(a))$ which verify \eqref{oordn} exist. In particular the processes 
$Y^1(a)$ and $K^1(a)$ are continuous. Next, in the same way, if 
for some $n\ge 1$, for any $a\in \rl$, the triple 
$(Y^n(a),Z^n(a),K^n(a))$ which satisfies \eqref{oordn} exists then for any $a\in \rl$, 
the triple $(Y^{n+1}(a),Z^{n+1}(a),K^{n+1}(a))$ 
satisfying \eqref{oordn} exits. This completes the proof of the well-posedness of the processes $(Y^n(a),Z^n(a),K^n(a))$ solution of \eqref{oordn}. Finally let us notice that the representation \eqref{ordrenn} holds true for $Y^n(a)$.
\end{proof} 
\def \g{\gamma}
\begin{Proposition}\lb{restric}For any $n\ge 0$ and $a\in \rl$, we have: $$
\forall t\in [T-\d,T],\,\,Y^n_t(a)=\E[\int_t^Tg(s,X_s+a)ds|\cf_t]=O^n_t(a).
$$    
\end{Proposition}
\begin{proof}If $n=0$, the equality holds true by definition of $Y^0(a)$. So  let us assume that for some $n\ge 1$, for any $a\in \rl$, the relation holds true. By its definition in \eqref{oordn}, the triple $(Y^{n}(a),Z^n({a}),K^{n}(a))$ verifies on $[T-\d, T]$:
\def \mx{\mbox}
\begin{equation}
\label{oordn2}
\left\{
  \begin{array}{ll}
        Y_{t}^{n}(a)=\int_{t}^{T}g(s,X_{s}+a)     \,ds+K_{T}^{n}(a)-K_{t}^{n}(a)-\int_{t}^{T}Z_{s}^{n}(a)\,dB_{s},\,\,T-\d\le t\leq T;\\
      Y_{t}^{n}(a)  \geq  O_{t}^{n}(a) \mbox{ for $t\in [T-\d, T]$ and } \Int_{T-\d}^{T}(Y_{t}^{n}(a)-O_{t}^{n}(a))\,dK_{t}^{n}(a)=0.
  \end{array}
\right.
\end{equation}
But for any $t\in [T-\d, T]$, 
$$O_{t}^{n}(a)=\E[\int_t^Tg(s,X_s+a)ds|\cf_t].$$
Next, once more for $t\in [T-\d, T]$,
let us set:
$$\tilde Y^n_t(a)=\E[\int_t^Tg(s,X_s+a)ds|\cf_t]. 
$$
Therefore, by the representation property of square integrable Brownian martingales, there exists a $\cp$-measurable process $(\tilde Z^n_t(a))_{T-\d\le t\le T}$  which is $dt\otimes d\P$-square integrable on $[T-\d,T]\times \Omega$ such that for any $t\in [T-\d, T]$,$$
\tilde Y^n_t(a)=\int_t^Tg(s,X_s+a)ds-\int_t^T\tilde Z^n_s(a)dB_s. 
$$It means that $(\tilde Y^n_t(a),\tilde Z^n_t(a),0)_{T-\d\le t\le T}$ is  also a solution for the reflected BSDE \eqref{oordn2}. But the solution of this latter is unique, therefore for any $t\in [T-\d,T]$, 
$$\tilde Y^n_t(a)=\E[\int_t^Tg(s,X_s+a)ds|\cf_t]=
Y^n_t(a)=O^n_t(a)
$$which is the desired result. \end{proof}
\begin{Remark}\lb{kna} By the uniqueness of the solution of \eqref{oordn2}, for any $n\ge 1$ and any $a\in \R^\ell$, it holds that $dK^n_t(a)=0$, for any $t\in[T-\Delta,T]$.
\end{Remark}
\def \cu {\mathcal{U}}
Next recall that $\mathfrak{X}=[\frac{T}{\d}]$ is the maximal amount of impulses that the controller is allowed to make due to the delay $\d$ and finite horizon $T$. So for $k=1,...,\mathfrak{X}$, let us define the finite set $\cu_k$ as follows:
$$
\cu_k:=\{a=\sum_{j=1}^{k}a_j, \mx{ where }(a_1,...,a_k)\in U^k\}.
$$\def \cv {\mathcal{V}}
For a fixed $k$, the $k$-tuple $(a_1,...,a_k)$ stands for the 
impulses implemented by the controller while $\sum_{j=1}^{k} a_j$ is the sum of the cumulative impulses up to the $k$-th one.  
\ms

 {Next let $\cv=\bigcup_{k=1,\mathfrak{X}} \cu_k$, $\t$ be a stopping time and $\xi$ an $\cf_\t$-random variable which takes its values in the finite subset $\mathcal{V}$ of $\R^{\ell}$. We set, for any $t\ge \t$,
\begin{equation}\label{yxi}
    Y_t^n (\xi) := \sum_{a \in \cv} Y_t^n (a) \1_{[\xi =a]}, Z_t^n (\xi) := \sum_{a \in \cv} Z_t^n (a) \1_{[\xi =a]},K_t^n (\xi) := \sum_{a \in \cv} K_t^n (a) \1_{[\xi =a]},  
\end{equation}
and finally
\def \x {\xi}
\begin{equation}\label{oxi}
    O_t^n (\xi) := \sum_{a \in \cv} O_t^n (a) \1_{[\xi =a]}.
\end{equation}
From the definition \eqref{oordn}, the triple $(Y_t^n (\xi), Z_t^n(\xi),K_t^n(\xi))_{t\in [\t,T]}$ verifies: $\P$-a.s. for any $\t~\le~t\le~T$, 
\begin{equation}
\label{oordn22}
\left\{
  \begin{array}{ll}
   \E[\sup_{t\in [\t,T]} |Y^{n}_t(\xi)|^2+K_T^n(\xi)^2]+
   \E[\int_\t^T|Z_t^n(\xi)|^2]dt]<\infty; \\   
       Y_{t}^{n}(\x)=\int_{t}^{T}g(s,X_{s}+\x)     \,ds+K_{T}^{n}(\x)-K_{t}^{n}(\x)-\int_{t}^{T}Z_{s}^{n}(\x)\,dB_{s};\\
      Y_{t}^{n}(\xi)  \geq  O_{t}^{n}(\xi) \mbox{ and } \Int_{\t}^{T}(Y_{t}^{n}(\x)-O_{t}^{n}(\x))\,dK_{t}^{n}(\x)=0.
  \end{array}
\right.
\end{equation}
\begin{Remark}\lb{aptd} As in Proposition \ref{restric}, 
for any $t\in [(T-\Delta)\vee \t, T]$,
$
Y_{t}^n({\x})=O_{t}^n({\x})=\E[\int_t^Tg(s,X_s+\xi)ds|\cf_t]$ and $dK^n_t(\xi)=0$.
\end{Remark}
\def \ci {\mathcal{I}}

\subsubsection{The optimal strategy}
\label{sec5}
We are going now to show that an optimal strategy exists. It is constructed with the help of processes $Y^n(a)$, $a\in \rl$, for some appropriate $n$. 
\ms

We first show:
\ms

\begin{Proposition}\label{optimal strategy}
For every $n \in \N^*$, there exists a strategy $\delta^n$ in $\mathcal{A}_{n}$ which satisfies:  
$$\sup_{\delta\in
\mathcal{A}_{n}}J(\delta)=J(\delta^{n})=Y_{0}^{n}(0).$$Thus $\delta^n$ is optimal for the impulse control problem when only $n$ impulses at most are allowed.  \label{approximation}
\end{Proposition}
\begin{proof} The proof will be divided into two steps.
\medskip

\noindent \textbf{Step 1:} Let us show that $Y_{0}^{n}(0)=J(\delta^n)$ where $\delta^n$ will be exhibited later on, 
as we progress through the proof. First recall that 
the triple $(Y^n(0),Z^n(0),K^n(0))$ verifies:
\begin{equation}
\label{oordn1}
\left\{
  \begin{array}{ll}
    Y^{n}(0)\in \mathcal{S}_c^{2}, \,Z^{n}(0)\in \mathcal{H}^{2,d} \mx{ and } K^{n}(0)\in \mathcal{S}_{i}^{2}; \\
    Y_{t}^{n}(0)=\int_{t}^{T}g(s,X_{s})     \,ds+K_{T}^{n}(0)-K_{t}^{n}(0)-\int_{t}^{T}Z_{s}^{n}(0)\,dB_{s},\,\,t\leq T;\\
      Y_{t}^{n}(0)  \geq  O_{t}^{n}(0), \,t\le T \mbox{ and } \Int_{0}^{T}(Y_{t}^{n}(0)-O_{t}^{n}(0))\,dK_{t}^{n}(0)=0.
  \end{array}
\right.
\end{equation}
Next let $\t^n_1$ be the following stopping time:
$$
\tau_1^n = \inf\{s \geq 0,\, Y_s^n(0) = O_s^n(0)\}.
$$
By Proposition \ref{restric}, $\P-a.s.$, $\t_1^n\le T-\d$. On the other hand since the process $(O_t^{n}(0))_{t\le T}$ is continuous on $[0,T]-\{T-\Delta\}$ and cannot have a negative jump at $T-\Delta$, then we have: 
\begin{eqnarray}
        Y_{0}^{n}(0)&=&\E[\Int_{0}^{\tau_{1}^{n}} g(s,X_{s}) \,ds + Y^n_{\tau_1^{n}}(0)]  \\
    &=& \E\big[\Int_{0}^{\tau_{1}^{n}} g(s,X_{s}) \,ds \nonumber +\1_{[\tau_1^n < T-\Delta]}O^n_{\tau_1^{n}}(0)+\1_{[\tau_1^n= T-\Delta]} \E[\int_{\tau_1^{n}}^T g(s,X_s) ds|\F_{\tau_1^n}]\big].
    \end{eqnarray}\label{y0n}
But \begin{eqnarray*}
O^n_{\tau_1^{n}}(0)=\E\big[\Int_{\tau_1^{n}}^{(\tau_1^{n}+\Delta)\wedge T}g(s,X_{s})\,ds|\mathcal{F}_{\tau_1^{n}}\big]
         +\1_{[\tau_1^{n}< T-\Delta]} \E[\max_{\beta\in U} \{-\psi(\beta)+Y_{\tau_1^{n}+\Delta}^{n-1}(\beta)\}|\mathcal{F}_{\tau_1^{n}}].
\end{eqnarray*}
\noindent So for $l\in \ci=\{1,\ldots,p\}$, let $
\Gamma^{1,n}_l$ be the following set:

\begin{equation*}
\Gamma^{1,n}_l= \bigcap_{ 1\leq j \leq p,\, j\neq l}\{-\psi(a_l)+Y _{\tau^n_{1}+\Delta}^{n-1}(a_l) \geq -\psi(a_j)+Y _{\tau^n_{1}+\Delta}^{n-1}(a_j)\}.
\end{equation*}
\def \txt{\text}
\def \tbar{\tilde}\\
Next we define $\tilde{\Gamma}^{1,n}_l$, $l\in {\cal I}$, by: 
\begin{equation}\tbar{\Gamma}^{1,n}_1 =\Gamma^{1,n}_1 \txt{ and for }l\ge 2, \tbar{\Gamma}^{1,n}_l=\Gamma^{1,n}_l\setminus [\Gamma^{1,n}_1 \cup \ldots \Gamma^{1,n}_{l-1}].\lb{ensgamma}\end{equation}
Those sets $\tbar{\Gamma}^{1,n}_l$, $l\in \ci$, belong to $\F_{\tau_1^n + \Delta}$, are pairwise disjoint and on 
$\tbar{\Gamma}^{1,n}_l $, if non empty, $a_l\in U$ realizes the maximum of 
$$\max_{a\in U}(-\psi(a)+Y _{\tau^n_{1}+\Delta}^{n-1}(a)).$$
Finally let us set 
\begin{equation}\label{beta1}
    \beta_1^n = \sum_{1\leq l \leq p} a_l \1_{\tbar{\Gamma}^{1,n}_l}.
\end{equation}
The random variable $\beta^n_1$ is $\F_{\tau_1^n + \Delta}-$measurable and verifies:
$$\max_{a\in U}(-\psi(a)+Y _{\tau^n_{1}+\Delta}^{n-1}(a))=
-\psi(\beta^n_1)+Y _{\tau^n_{1}+\Delta}^{n-1}(\beta^n_1).$$
Therefore 
\begin{eqnarray*}
O^n_{\tau_1^{n}}(0)=\E\big[\Int_{\tau_1^{n}}^{(\tau_1^{n}+\Delta)\wedge T}g(s,X_{s})\,ds|\mathcal{F}_{\tau_1^{n}}\big]
         +\1_{[\tau_1^{n}< T-\Delta]} \E[(-\psi(\beta^n_1)+Y _{\tau^n_{1}+\Delta}^{n-1}(\beta^n_1))|\mathcal{F}_{\tau_1^{n}}]
\end{eqnarray*}
and then
\def \b {\beta}
\def \bn {\beta^n}
\begin{eqnarray}\label{y0n2}
    Y_{0}^{n}(0)
    &=& \E\{\Int_{0}^{\tau_{1}^{n}} g(s,X_{s}) \,ds \nonumber +\1_{[\tau_1^n < T-\Delta]}\E\big[\Int_{\tau_1^{n}}^{(\tau_1^{n}+\Delta)\wedge T}g(s,X_{s})\,ds|\mathcal{F}_{\tau_1^{n}}\big]
       \nonumber  \\&{}&\, +\1_{[\tau_1^{n}< T-\Delta]} \E[(-\psi(\beta^n_1)+Y _{\tau^n_{1}+\Delta}^{n-1}(\beta^n_1))|\mathcal{F}_{\tau_1^{n}}]+\1_{[\tau_1^n= T-\Delta]} \E[\int_{\tau_{1}^{n}}^T g(s,X_s) ds|\F_{\tau_1^n}]\} \nonumber\\
          &=& \E\{\Int_{0}^{\tau_{1}^{n}+\d} g(s,X_{s}) \,ds  +\1_{[\tau_1^{n}< T-\Delta]} (-\psi(\beta^n_1))+Y _{\tau^n_{1}+\Delta}^{n-1}(\beta^n_1)\}
    \end{eqnarray}
since $Y _{\tau^n_{1}+\Delta}^{n-1}(\b^n_1)=\1_{[\tau_1^{n}< T-\Delta]}.Y _{\tau^n_{1}+\Delta}^{n-1}(\bn_1)$ because 
$Y _T^{n-1}(\b^n_1)=0$, $\P$-a.s. This is for the first step. 

Next we deal with the quantity $Y _{\tau^n_{1}+\Delta}^{n-1}(\beta^n_1)$. First, let us point out that by \eqref{oordn22}, the triple 
$(Y_t^{n-1} (\b^n_1), Z_t^{n-1}(\b^n_1),K_t^{n-1}(\b^n_1))_{t\in [\tau^n_{1}+\Delta,T]}$ verifies: $\P$-a.s. for any $\tau^n_{1}+\Delta\le t\le T$, 
\begin{equation}
\label{oordn21}
\left\{
  \begin{array}{ll}
   Y_{t}^{n-1}(\b^n_1)=\int_{t}^{T}g(s,X_{s}+\b^n_1)     \,ds+K_{T}^{n-1}(\b^n_1)-K_{t}^{n-1}(\b^n_1)-\int_{t}^{T}Z_{s}^{n-1}(\b^n_1)\,dB_{s};\\
      Y_{t}^{n-1}(\b^n_1)  \geq  O_{t}^{n-1}(\b^n_1) \mbox{ and } \Int_{\tau^n_{1}+\Delta}^{T}(Y_{t}^{n-1}(\b^n_1)-O_{t}^{n-1}(\b^n_1))\,dK_{t}^{n-1}(\b^n_1)=0.
  \end{array}
\right.
\end{equation}
Let $D_1^n= \{ \tau_1^n \leq T-2\Delta\}$ and $\bar D_1^n$ its complement, which belong to $\F_{\tau_1^n}$, and let $\gamma_2^n$ and $\t^n_2$ be the following stopping times:
$$
\gamma_2^n = \inf\{s \geq (\tau_1^n+\d), \, Y_s^{n-1}(\b^n_1) = O_s^{n-1}(\b^n_1)\},
$$
and 
$$
\tau_2^n = \gamma_2^n \1_{D_1^n} + T \1_{\bar{D_1^n}}.
$$
Since on $\bar{D_1^n}$, $\tau_1^n+\d > T-\Delta $, then by virtue of  Proposition \ref{restric} and Remark \ref{aptd}, we have:
$$
  \1_{\bar{D_1^n}} Y_{\tau_1^n+\d}^{n-1}(\bn_1) = \1_{\bar{D_1^n}} \E[\int_{\tau_1^n+\d}^Tg(s,X_s+\beta_1^n)ds|\cf_{\tau_1^n+\d}].
$$
Now, once more by Proposition \ref{restric}, $\P$-a.s. $\gamma_2^n \1_{D_1^n} \le T-\d$ since $Y_{T-\d}^{n-1}(\b^n_1) = O_{T-\d}^{n-1}(\b^n_1)$. On the other hand, since $K^{n-1}(\beta_1^{n-1})$ is constant on $[\t^n_1+\d, \t^n_2]\cap D^n_1$, 
we have:
\def \nn{\nonumber}
\begin{eqnarray}
       \1_{D_1^n} Y_{\tau_1^n+\d}^{n-1}(\bn_1)&=& \E\{\1_{D_1^n} \Int_{\tau_1^n+\d}^{\tau_{2}^{n}} g(s,X_{s}+\bn_1)ds + \1_{D_1^n}Y_{\tau_2^n}^{n-1}(\bn_1)|\F_{\tau_1^n+\d}\} \nonumber \\
    &=&   \E\big\{\1_{D_1^n} \Int_{\tau_1^n+\d}^{\tau_{2}^{n}} g(s,X_{s}+\bn_1)ds
    + \1_{D_1^n \cap [\tau_2^n < T-\Delta]}O^{n-1}_{\tau_2^{n}}(\bn_1) \nonumber \\&{}&\qq+\1_{D_1^n \cap [\tau_2^n= T-\Delta]} \E\{\int_{\tau_2^{n}}^T g(s,X_s+\bn_1) ds|\F_{\tau_2^n}\}|\F_{\tau_1^n+\d}\big\}.
    \lb{eq1}\end{eqnarray}
Next we focus on $\1_{D_1^n \cap [\tau_2^n < T-\Delta]}O^{n-1}_{\tau_2^{n}}(\bn_1)$. Recall that 
$$O_{\tau_2^n}^{n-1}(\beta_1^n)=
   \sum_{a \in U} \1_{\{\beta_1^n =a\}}  O_{\tau_2^n}^{n-1}(a)$$
and for any $a\in \R^{\ell}$,
 \begin{eqnarray*}
      O_{\tau_2^n}^{n-1}(a):=\E\big[\Int_{\tau_2^n}^{(\tau_2^n+\Delta)\wedge T}g(s,X_{s}+a)\,ds|\mathcal{F}_{\tau_2^n}\big]
         +\1_{[\tau_2^n< T-\Delta]} \E[\max_{\beta\in U} \{-\psi(\beta)+Y_{\tau_2^n+\Delta}^{n-2}(a+\beta)\}|\mathcal{F}_{\tau_2^n}].
\end{eqnarray*}
Therefore, since on $D_1^n$, $\t^n_1+\d\le\tau_2^n \leq T-\Delta$ and $\beta_1^n$ is $\cf_{\t^n_1+\d}$-measurable, we have:
\def \qq {\qquad}
\begin{eqnarray*}
    \1_{D_1^n} O_{\tau_2^n}^{n-1}(\beta_1^n)
&=&\1_{D_1^n}\E\big[\Int_{{\tau_2^n}}^{({\tau_2^n}+\Delta)}g(s,X_{s}+\beta_1^n)\,ds|\mathcal{F}_{{\tau_2^n}}\big]
        \\&+&\1_{D_1^n \cap [{\tau_2^n}< T-\Delta]} \E[\sum_{a \in U} \1_{\{\beta_1^n =a\}} \{\max_{\beta\in U} \{-\psi(\beta)+Y_{{\tau_2^n}+\Delta}^{n-2}(a+\beta)\}\}|\mathcal{F}_{{\tau_2^n}}].
        \end{eqnarray*} 
        But,  
         \begin{eqnarray*}\sum_{a \in U} \1_{\{\beta_1^n =a\}} \{\max_{\beta\in U} \{-\psi(\beta)+Y_{{\tau_2^n}+\Delta}^{n-2}(a+\beta)\}\}&=& \max_{\beta\in U} \{\sum_{a \in U} \1_{\{\beta_1^n =a\}} (-\psi(\beta)+Y_{{\tau_2^n}+\Delta}^{n-2}(a+\beta))\}\\ &=&\max_{\beta\in U} \{ -\psi(\beta)+Y_{{\tau_2^n}+\Delta}^{n-2}(\beta_1^n+\beta)\}.
         \end{eqnarray*}
        Then
        \begin{eqnarray}
     \1_{D_1^n} O_{\tau_2^n}^{n-1}(\beta_1^n)&=& \1_{D_1^n} \E\big[\Int_{{\tau_2^n}}^{({\tau_2^n}+\Delta)}g(s,X_{s}+\beta_1^n)\,ds|\mathcal{F}_{{\tau_2^n}}\big]\nonumber
        \\&+&\1_{D_1^n \cap [{\tau_2^n}< T-\Delta]} \E[\max_{\beta\in U} \{ -\psi(\beta)+Y_{{\tau_2^n}+\Delta}^{n-2}(\beta_1^n+\beta)\}|\mathcal{F}_{{\tau_2^n}}].
        \lb{eq2}
\end{eqnarray}
Now, for every $l\in \ci$, let 
\begin{equation*}
\Gamma^{2,n}_l = \bigcap_{ 1\leq j \leq p}\{-\psi(a_l)+Y _{\tau_{2}^n+\Delta}^{n-2}(\beta_{1}^n+a_l) \geq -\psi(a_j)+Y _{\tau_{2}^n+\Delta}^{n-2}(\beta_{1}^n+a_j)\},
\end{equation*}

$\tbar{\Gamma}^{2,n}_1 = {\Gamma}^{2,n}_1$, for $l\ge 2$,
 $ \tbar{\Gamma}^{2,n}_l =\Gamma^{2,n}_l \setminus [\Gamma^{2,n}_1 \cup \ldots \Gamma^{2,n}_{l-1}]$,
and $$
\beta_2^n = \1_{D_1^n}\big(\sum_{1\leq l \leq p} a_l \1_{\tbar{\Gamma}^{2,n}_l}\big)+a_1 \1_{\overline{D_1^n}}.$$
The last random variable is $\mathcal{F}_{\tau_2^n + \Delta}-$measurable since the sets $\tbar{\Gamma}^{2,n}_l$ belong to $\mathcal{F}_{\tau_2^n + \Delta}$ and its value on ${\overline{D_1^n}}$ is irrelevant since, taking into account the required delay, the interventions after $T$ are not allowed. For that reason, we  take $a_1$, but we can consider any other element of $U$. Now, if $ \tbar{\Gamma}^{2,n}_l $ is not empty, we have   
$$
\max_{\beta\in U} \{ -\psi(\beta)+Y_{{\tau_2^n}+\Delta}^{n-2}(\beta_1^n+\beta)\} = -\psi(a_l)+Y_{{\tau_2^n}+\Delta}^{n-2}(\beta_1^n+a_l) \text{ on } \tbar{\Gamma}^{2,n}_l, 
$$
which means that
\be \lb{eqmax}
\max_{\beta\in U} \{ -\psi(\beta)+Y_{{\tau_2^n}+\Delta}^{n-2}(\beta_1^n+\beta)\} = -\psi(\beta_2^n)+Y_{{\tau_2^n}+\Delta}^{n-2}(\beta_1^n+\beta_2^n).
\ee
Therefore by \eqref{eq1}, \eqref{eq2} and \eqref{eqmax},
\begin{eqnarray*}
   \1_{D_1^n} Y^{n-1}_{\tau_1^n + \Delta} (\beta_1^n)&= &\E[\1_{D_1^n} \Int_{\tau_1^n + \Delta}^{{\tau_2^n}+\Delta}g(s,X_{s}+\beta_1^n)\,ds
        \\&{}&+ \1_{D_1^n \cap [{\tau_2^n}< T-\Delta]} \{-\psi(\beta^n_2)+Y_{{\tau_2^n}+\Delta}^{n-2}(\beta_1^n+\beta^n_2)\}    
    \big|\F_{\tau_1^n + \Delta}]
     \end{eqnarray*}
which implies, 
\begin{eqnarray*}
   Y^{n-1}_{\tau_1^n + \Delta} (\beta_1^n)&=& \E[\1_{\bar{D_1^n}} \int_{\tau_1^n+\d}^Tg(s,X_s+\beta_1^n)ds+ \1_{D_1^n} \Int_{\tau_1^n + \Delta}^{{\tau_2^n}+\Delta}g(s,X_{s}+\beta_1^n)\,ds\\
        &+& \1_{D_1^n \cap [{\tau_2^n}< T-\Delta]} \{-\psi(\beta^n_2)+Y_{{\tau_2^n}+\Delta}^{n-2}(\beta_1^n+\beta^n_2)\}    
    \big|\F_{\tau_1^n + \Delta}].
     \end{eqnarray*}
By plugging the last equality in (\ref{y0n2}) and taking into account  the fact that $D_1^n\subset \{\tau_1^n < T-\d\}$, $\1_{D^n_1}+\1_{\overline{D^n_1}}=1$ and $Y_{{T}}^{n-2}(\beta_1^n+\beta^n_2)=0$, we obtain
\begin{eqnarray*}
    Y_{0}^{n}(0)
             &=& \E\big[\1_{\bar{D}_1^n} \{\Int_{0}^{\tau_{1}^{n}+\d} g(s,X_{s}) \,ds  + \Int_{\tau_1^n + \Delta}^{T}g(s,X_{s}+\beta_1^n)\,ds -\1_{[\tau_1^{n}< T-\Delta]} \psi(\beta^n_1)\}\\ &+& \1_{D_1^n} \{\Int_{0}^{\tau_{1}^{n}+\d} g(s,X_{s}) \,ds  + \Int_{\tau_1^n + \Delta}^{({\tau_2^n}+\Delta)}g(s,X_{s}+\beta_1^n)\,ds -\psi(\beta^n_1)\}
        \\&+& \1_{D_1^n \cap [{\tau_2^n}< T-\Delta]} (-\psi(\beta^n_2))+\1_{D_1^n}Y_{{\tau_2^n}+\Delta}^{n-2}(\beta_1^n+\beta^n_2) \big].
    \end{eqnarray*}
    \noindent 
Now, for any $k\in \{2,\cdots ,n\},$ 
let $D_{k-1}^n = \{\tau_{k-1}^n \le T-2\Delta\}$,  
$$\gamma_{k}^{n}=\inf\{s\geq (\tau_{k-1}^{n}+\Delta), \,\, Y_{s}^{n-k+1}( \beta_{1}^{n}+\cdots+\beta_{k-1}^{n})=O_{s}^{n-k+1}( \beta_{1}^{n}+\cdots+\beta_{k-1}^{n})\}
$$
and 
$$\tau_{k}^{n}= \gamma_k^n \1_{D_{k-1}^n} + T \1_{\overline{D_{k-1}^n}}.$$ 
We can easily see that the sequence of sets 
${(D_k^n)}_{k=1,...,n-1}$ is decreasing, the set $D_k^n$ is \\$\F_{\tau_k^n}-$measurable and the family 
$$\{\overline{D_1^n},\, D_{n-1}^n,\,D_i^n \cap \overline{D_{i+1}^n};\, 1 \leq i \leq n-2\}$$ forms a complete system. So let $k\in \{3,\cdots,n\}$, $l\in \cal I$, and let us set:
\begin{equation*}
\Gamma^{k,n}_l =\{-\psi(a_l)+Y _{\tau_{k}^n+\Delta}^{n-k}(\beta_{1}^n+\ldots+\beta_{k-1}^n+a_l) \geq -\psi(a_j)+Y _{\tau_{k}^n+\Delta}^{n-k}(\beta_{1}^n+\ldots+\beta_{k-1}^n+a_j), \,\forall j=1,\dots,p\}
\end{equation*}
and $\tbar{\Gamma}^{k,n}_1  =\Gamma^{k,n}_1$ and for $l\ge 2, \tbar{\Gamma}^{k,n}_l  =\Gamma^{k,n}_l  \setminus \bigcup_{ m=1}^{l-1} \Gamma^{k,n}_m$.
Those sets $\tbar{\Gamma}^{k}_l$, $l\in \ci$, belong to $\F_{\tau_k^n + \Delta}$, are pairwise disjoint  and on 
$\tbar{\Gamma}^{k}_l  $, if non empty, $a_l$ realizes the maximum over $ U$ of 
$$(-\psi(a)+Y _{\tau_{k}^n+\Delta}^{n-k}(\beta_{1}^n+\ldots+\beta_{k-1}^n+a)).$$
So let $$
\beta_k^n = \{\sum_{1\leq l \leq p} a_l \1_{\tbar{\Gamma}^{k,n}_l}\}\1_{D_{k-1}^n} +a_1 \1_{\overline{D_{k-1}^n}}.
$$
On $D_{k-1}^n$, we have:
$$\max_{a \in U}(-\psi(a)+Y _{\tau_{k}^n+\Delta}^{n-k}(\beta_{1}^n+\ldots+\beta_{k-1}^n+a))= -\psi(\beta_k^n)+Y _{\tau_{k}^n+\Delta}^{n-k}(\beta_{1}^n+\ldots+\beta_{k-1}^n+\beta_k^n).$$
 By iterating the same  reasoning as above, we obtain: 
{\small{
\begin{eqnarray*}
  Y_{0}^{n}(0)    &=& \E\Big[\1_{\overline{D_1^n}} \{\Int_{0}^{\tau_{1}^{n}+\Delta } g(s,X_{s}) \,ds + \Int_{\tau_{1}^{n}+\Delta}^{T} g(s,X_{s}+\beta_1^n) \,ds +\1_{\{\tau_{1}^n <T-\Delta\}}  (-\psi(\beta_{1}^n))\}\\&+& 
  \sum_{i=1}^{k-1} \1_{D_i^n \cap \overline{D_{i+1}^n}} \{\Int_{0}^{\tau_{i+1}^{n}+\Delta } g(s,X_{s}^{\delta^n}) \,ds+\Int_{\tau_{i+1}^{n}+\Delta }^T g(s,X_{s} + \beta_1^n+\ldots+\beta_{i+1}^n) \,ds -  \sum_{j=1}^{i+1} \psi(\beta^n_j) \1_{[\tau_{j}^n < T-\Delta]} \}\\&+& \1_{D^n_k} \{\Int_{0}^{\tau_{k+1}^{n}+\Delta } g(s,X_{s}^{\delta^n}) \,ds-\sum_{i=1}^{k+1} \psi(\beta_i) \1_{[\tau_{i}^n < T-\Delta]} + Y_{\tau_{k+1}^n+\Delta}^{n-k-1} (\beta_1^n+\ldots+\beta_{k+1}^n) \}
 \Big]
\end{eqnarray*}
}}
where for any $s<\t_{k+1}^n+\d$, $$X_s^{\delta^n}=X_s1_{\{0\le s<\t^n_{1}+\d\}}+\sum_{i=1}^k(X_s+\b^n_1+\b^n_2+\dots+\b^n_{i})1_{\{\t^n_i+\d\le s<\t^n_{i+1}+\d\}}.$$
In particular, for $k=n-1$, we have
 {\small{
 \begin{eqnarray*}
  Y_{0}^{n}(0)    &=& \E\Big[\1_{\overline{D_1^n}} \{\Int_{0}^{\tau_{1}^{n}+\Delta } g(s,X_{s}) \,ds + \Int_{\tau_{1}^{n}+\Delta}^{T} g(s,X_{s}+\beta_1^n) \,ds +\1_{\{\tau_{1}^n <T-\Delta\}}  (-\psi(\beta_{1}^n))\}\\&+& 
  \sum_{i=1}^{n-2} \1_{D_i^n \cap \overline{D_{i+1}^n}} \{\Int_{0}^{\tau_{i+1}^{n}+\Delta } g(s,X_{s}^{\delta^n}) \,ds +\Int_{\tau_{i+1}^{n}+\Delta }^T g(s,X_{s} + \beta_1^n+\ldots+\beta_{i+1}^n) \,ds -  \sum_{j=1}^{i+1} \psi(\beta^n_j) \1_{[\tau_{j}^n < T-\Delta]} \}\\&+& \1_{D^n_{n-1}} \{\Int_{0}^{\tau_{n}^{n}+\Delta } g(s,X_{s}^{\delta^n}) \,ds-\sum_{i=1}^{n} \psi(\beta_i)\1_{[\tau_{i}^n < T-\Delta]} + Y_{\tau_{n}^n+\Delta}^{0} (\beta_1^n+\ldots+\beta_{n}^n) \}
 \Big].   
\end{eqnarray*}
}}
However, 
\begin{eqnarray*}
    Y_{\tau_{n}^n + \Delta}^{0}(\beta_1^n+\ldots+\beta_{n}^n)=\E[\Int_{\tau_{n}^n + \Delta}^{T} g(s,X_s+\beta_1^n+\ldots+\beta_{n}^n) ds |\F_{\tau_{n}^n + \Delta}]
\end{eqnarray*}
which implies that
\small{
 \begin{eqnarray}
  Y_{0}^{n}(0)    &=& \E\Big[\1_{\overline{D_1^n}} \{\Int_{0}^{\tau_{1}^{n}+\Delta } g(s,X_{s}) \,ds + \Int_{\tau_{1}^{n}+\Delta}^{T} g(s,X_{s}+\beta_1^n) \,ds +\1_{\{\tau_{1}^n <T-\Delta\}}  (-\psi(\beta_{1}^n))\}\nn\\&+& 
  \sum_{i=1}^{n-2} \1_{D_i^n \cap \overline{D_{i+1}^n}} \{\Int_{0}^{\tau_{i+1}^{n}+\Delta } g(s,X_{s}^{\delta^n}) \,ds +\Int_{\tau_{i+1}^{n}+\Delta }^T g(s,X_{s} + \beta_1^n+\ldots+\beta_{i+1}^n) \,ds -  \sum_{j=1}^{i+1} \psi(\beta^n_j) \1_{[\tau_{j}^n < T-\Delta]} \}\nn\\&+& \1_{D_{n-1}^n} \{\Int_{0}^{T} g(s,X_{s}^{\delta^n}) \,ds-\sum_{i=1}^{n} \psi(\beta^n_i) \1_{[\tau_{i}^n < T-\Delta]} \}
 \Big]  \lb{eqyno}
\end{eqnarray}
}where $\delta^n=(\t^n_k,\bn_k)_{k=1,n}$ and $X^{\dl^n}$ is defined as in \eqref{xd}. Let us now show that $\dl^n$ is admissible. First note that $\t^n_1\le T-\d$. Next for $k\in \{1,...,n-1\}$, on $D^n_k$, $\gamma^n_{k+1} \le T-\d$, then it follows from its definition that $\t^n_{k+1}=\gamma^n_{k+1}\le T$. Now if $\omega \in \bar D^n_{k}$, 
$\t^n_{k+1}(\omega)=T$. Thus 
$\t^n_{k+1}(\omega)\le T$, $\P$-a.s.
Next on $D^n_k$, 
$\t^n_{k}+\d\le \t^n_{k+1}=\g^n_{k+1}\le T-\d$ which implies that on $D^n_k$, 
$\t^n_{k+1}\ge \min\{\t^n_{k}+\d, T\}$ while on  $\bar D^n_k$,
$\t^n_{k+1}=T$. It follows that 
$\t^n_{k+1}\ge \min\{\t^n_{k}+\d, T\}$, $\P$-a.s. for any  $k\in \{1,...,n-1\}$. Finally for any $k\in \{1,\ldots,n\}$, $\b_k^{n} \in \F_{{\tau}_k^n + \d}$, then $\dl^n$ is an admissible strategy. Next \begin{eqnarray}
    J(\delta^n)&=&\E[\{\1_{\overline{D_1^n}}+
\sum_{i=1}^{n-2} \1_{D_i^n \cap \overline{D_{i+1}^n}}+\1_{D_{n-1}^n}\}
\{\Int_{0}^{T} g(s,X^{\delta^n}_{s}) \,ds - \sum_{k =1,n} \psi(\beta_{k}^n)\1_{\{\tau_{k}^n <T-\Delta\}} \}
]\nn\\
&=&\text{the right-hand side of \eqref{eqyno}},\nn
\end{eqnarray}
thus $J(\delta^n)=Y^n_0(0).$
\medskip

\noindent \textbf{Step 2:} $J(\delta^{n})\geq J(\bar{\delta}_{n})$ for any $\bar{\delta}_{n}\in\mathcal{A}_{n}$. 
\medskip
\def \bt {\bar{\tau}}
\def \btn{\bt^n}

\noindent Let $\bar{\delta}_{n}=(\bar{\tau}_{k}^{n},\bar{\beta}_{k}^{n})_{k\ge 1}$ be an admissible strategy of $\ca_n$. Therefore $\bt^n_{k}=T$, for any $k\ge n+1$, $\btn_k\le \btn_{k+1}$ and for any $k\ge 1$, $T\ge \btn_{k+1}\ge \min\{T,\btn_{k}+\d\}$. On the other hand  
$\bar{\beta}_{k}^{n}$ is $\Fc_{(\bt^n_k+\d)\wedge T}$-measurable.  
\medskip
    

\noindent The characterization \eqref{ordrenn} of $Y^n(0)$ implies that: 
\begin{eqnarray}\label{y0nexp}
Y_{0}^{n}(0)&\geq& \E[\Int_{0}^{\bar{\tau}_{1}^{n}} g(s,X_{s}) \,ds +O^n_{\bar{\tau}_{1}^{n}}(0)]\nonumber\\
&\geq& \E\big[\Int_{0}^{\bar{\tau}_{1}^{n}}g(s,X_{s})\,ds+\1_{[\bar{\tau}_{1}^{n}< T-\Delta]} \E[\Int_{\bar{\tau}_{1}^n}^{\bar{\tau}_{1}^{n}+\Delta}g(s,X_{s})\,ds|\Fc_{\bar{\tau}_{1}^n}]\nonumber \\&{}&+
\1_{[\bar{\tau}_{1}^{n}< T-\Delta]} \E[\big(-\psi(\bar{\beta}_{1}^{n})+Y_{\bar{\tau}_{1}^{n}+\Delta}^{n-1}(\bar{\beta}_{1}^{n})\big)|\Fc_{\bar{\tau}_{1}}] + \1_{[T-\Delta \leq \bar{\tau}_{1}^{n}\leq T]} \E[\Int_{\bar{\tau}_{1}^{n}}^{T}g(s,X_{s})\,ds|\Fc_{\bar{\tau}_{1}^{n}}] \big]\nonumber\\
&\geq& \E \big[\Int_{0}^{(\bar{\tau}_{1}^{n}+\Delta)\wedge T}g(s,X_{s})\,ds+
\1_{[\bar{\tau}_{1}^{n}< T-\Delta]} \big\{-\psi(\bar{\beta}_{1}^{n})+Y_{\bar{\tau}_{1}^{n}+\Delta}^{n-1}(\bar{\beta}_{1}^{n})\big\} \big].
\end{eqnarray}
 Next by \eqref{oordn22} and the equation satisfied by $(Y_t^{n-1}(\bar{\beta}_{1}^{n}))_{t\ge \bar \t^n_1+\d}$, we have:
\begin{eqnarray*}
\1_{[\bar{\tau}_{1}^{n}< T-\Delta]}\,Y_{\bar{\tau}_{1}^{n}+\Delta}^{n-1}(\bar{\beta}_{1}^{n})
&\geq&\1_{[\bar{\tau}_{1}^{n}< T-\Delta]}\,\E\big[\Int_{\bar{\tau}_{1}^{n}+\Delta}^{\bar{\tau}_{2}^{n}}g(s,X_{s}+\bar{\beta}_{1}^{n})\,ds+ O_{\bar{\tau}_{2}^{n}}^{n-1}(\bar{\beta}_{1}^{n})|\mathcal{F}_{\bar{\tau}_{1}^{n}+\Delta}\big]\\
&=&\E\big[\1_{[\bar{\tau}_{1}^{n}< T-\Delta]}\Int_{\bar{\tau}_{1}^{n}+\Delta}^{\bar{\tau}_{2}^{n}}g(s,X_{s}+\bar{\beta}_{1}^{n})\,ds+\1_{[\bar{\tau}_{1}^{n}< T-\Delta]} O_{\bar{\tau}_{2}^{n}}^{n-1}(\bar{\beta}_{1}^{n})|\mathcal{F}_{\bar{\tau}_{1}^{n}+\Delta}\big]\\
&\ge &\E\big[\1_{[\bar{\tau}_{1}^{n}< T-\Delta]}\Int_{\bar{\tau}_{1}^{n}+\Delta}^{\bar{\tau}_{2}^{n}}g(s,X_{s}+\bar{\beta}_{1}^{n})\,ds+\1_{[\bar{\tau}_{2}^{n}<T-\d]}O_{\bar{\tau}_{2}^{n}}^{n-1}(\bar{\beta}_{1}^{n})\\&&\qq\qq +\1_{{[\bar{\tau}_{1}^{n}< T-\Delta]} \cap [T-\d \leq \bar{\tau}_{2}^{n}\leq T]}
\E\big[\Int_{\bar{\tau}_{2}^{n}}^{T}g(s,X_{s}+\bar{\beta}_{1}^{n})ds|
\mathcal{F}_{\bar{\tau}_{2}^{n}}]
|\mathcal{F}_{\bar{\tau}_{1}^{n}+\Delta}\big].
\end{eqnarray*}The first inequality stems from the fact that the process 
$(K_t^{n-1}(\bar{\beta}_{1}^{n}))_{t\ge \bar \t_1+\d}$ is non-decreasing. Once more $\bar{\beta}_{1}^{n}$ is an $\F_{T\wedge(\bar{\tau}_{1}^{n}+\Delta)}-$ measurable random variable and $\bar{\beta}_{1}^{n} = \sum_{a \in U} a\1_{[\bar{\beta}_{1}^{n} =a]}$, then we have:
\begin{eqnarray*}
\1_{[\bar{\tau}_{2}^{n}<T-\d]}O_{\bar{\tau}_{2}^{n}}^{n-1}(\bar{\beta}_{1}^{n})&=& \sum_{a \in U}\1_{[\bar{\tau}_{2}^{n}<T-\d]} O_{\bar{\tau}_{2}^{n}}^{n-1}(a) \1_{[\bar{\beta}_{1}^{n} = a]}\\
&=& \sum_{a \in U} \{\E[\1_{[\bar{\tau}_{2}^{n}<T-\d]}\int_{\bar{\tau}_{2}^{n}}^{(\bar{\tau}_{2}^{n} + \Delta)} g(s,X_s+a) ds \\&&\qq +  \1_{[\bar{\tau}_{2}^{n}<T-\Delta]} \E[\max_{\beta \in U} (-\psi(\beta) + Y_{\bar{\tau}_{2}^{n} + \Delta}^{n-2} (a+\beta))|\F_{\bar{\tau}_{2}^{n}}] 
   \}\1_{[\bar{\beta}_{1}^{n} = a]}  \\
&\geq& \sum_{a \in U}\E[ \{\1_{[\bar{\tau}_{2}^{n}<T-\d]} \int_{\bar{\tau}_{2}^{n}}^{(\bar{\tau}_{2}^{n} + \Delta)} g(s,X_s+a) ds\\&&\qq  +  \1_{[\bar{\tau}_{2}^{n}<T-\Delta]}  (-\psi(\bar{\beta}_{2}^{n}) + Y_{\bar{\tau}_{2}^{n} + \Delta}^{n-2} (a+\bar{\beta}_{2}^{n}))
   \}|\F_{\bar{\tau}_{2}^{n}}]\1_{[\bar{\beta}_{1}^{n} = a]} \\
   &\ge& \E[ \{\1_{[\bar{\tau}_{2}^{n}<T-\d]}\int_{\bar{\tau}_{2}^{n}}^{(\bar{\tau}_{2}^{n} + \Delta)} g(s,X_s+\bar{\beta}_{1}^{n}) ds \\&{}&\qq  +  \1_{[\bar{\tau}_{2}^{n}<T-\Delta]}  (-\psi(\bar{\beta}_{2}^{n}) )+ Y_{\bar{\tau}_{2}^{n} + \Delta}^{n-2} (\bar{\beta}_{1}^{n}+\bar{\beta}_{2}^{n})
   \}|\F_{\bar{\tau}_{2}^{n}}].
\end{eqnarray*}
Then
{\small{\begin{eqnarray*}
&{}&\1_{[\bar{\tau}_{1}^{n}< T-\Delta]}\,Y_{\bar{\tau}_{1}^{n}+\Delta}^{n-1}(\bar{\beta}_{1}^{n})\geq\E\big[\1_{[\bar{\tau}_{1}^{n}< T-\Delta]}\,\Int_{\bar{\tau}_{1}^{n}+\Delta}^{\bar{\tau}_{2}^{n}}g(s,X_{s}+\bar{\beta}_{1}^{n})\,ds+\1_{[\bar{\tau}_{1}^{n}< T-\Delta  \leq \bar{\tau}_{2}^{n}\leq T]}
\E\big[\Int_{\bar{\tau}_{2}^{n}}^{T}g(s,X_{s}+\bar{\beta}_{1}^{n})ds|
\mathcal{F}_{\bar{\tau}_{2}^{n}}]\\&{}&\qq\qq+
\E[\1_{[\bar{\tau}_{2}^{n}<T-\d]}\int_{\bar{\tau}_{2}^{n}}^{(\bar{\tau}_{2}^{n} + \Delta)} g(s,X_s+\bar{\beta}_{1}^{n}) ds +  \1_{[\bar{\tau}_{2}^{n}<T-\Delta]}  \{-\psi(\bar{\beta}_{2}^{n}) + Y_{\bar{\tau}_{2}^{n} + \Delta}^{n-2} (\bar{\beta}_{1}^{n}+\bar{\beta}_{2}^{n})\}
   |\F_{\bar{\tau}_{2}^{n}}]|\mathcal{F}_{\bar{\tau}_{1}^{n}+\Delta}\big]\\
&{}&\qq\qq\qq\ge \E\big[\1_{[\bar{\tau}_{1}^{n}< T-\Delta]}\Int_{\bar{\tau}_{1}^{n}+\Delta}^{(\bar{\tau}_{2}^{n}+\d )\wedge T}g(s,X_{s}+\bar{\beta}_{1}^{n})\,ds+  \1_{[\bar{\tau}_{2}^{n}<T-\Delta]}\{-\psi(\bar{\beta}_{2}^{n}) + Y_{\bar{\tau}_{2}^{n} + \Delta}^{n-2} (\bar{\beta}_{1}^{n}+\bar{\beta}_{2}^{n})\}|\mathcal{F}_{\bar{\tau}_{1}^{n}+\Delta}\big].
\end{eqnarray*}}}
Therefore going back to \eqref{y0nexp} to obtain, $$\begin{array}{lll}
Y_{0}^{n}(0)&\geq&\E\Big[\Int_{0}^{\bar{\tau}_{1}^{n}+\Delta}g(s,X_{s})\,ds+\Int_{(\bar{\tau}_{1}^{n}+\Delta)\wedge T}^{(\bar{\tau}_{2}^{n}+\Delta)\wedge T}g(s,X_{s}+\bar{\beta}_{1}^{n})\,ds+\1_{[\bar{\tau}_{1}^{n} <T-\Delta]}(-\psi(\bar{\beta}_{1}^{n}))\\&+&\1_{[\bar{\tau}_{2}^{n}<T-\Delta]}(-\psi(\bar{\beta}_{2}^{n}))+\1_{[\bar{\tau}_{2}^{n}< T-\Delta]}Y_{\bar{\tau}_{2}^{n}+\Delta}^{n-2}(\bar{\beta}_{1}^{n}+\bar{\beta}_{2}^{n})\Big].
\end{array}$$
Repeating now this procedure as many times as necessary to obtain: For any $n\geq 1$,
$$\begin{array}{lll}
Y_{0}^{n}(0)&\geq&\E[\Int_{0}^{\bar{\tau}_{1}^{n}+\Delta} g(s,X_{s}) \;ds +\Sum_{1 \leq k\leq n-1}\Int_{(\bar{\tau}_{k}^{n}+\Delta)\wedge T}^{(\bar{\tau}_{k+1}^{n}+\Delta)\wedge T}g(s,X_{s}+\bar{\beta}_{1}^{n}+\ldots+\bar{\beta}_{k}^{n})\,ds \\
&+&\Sum_{1\leq k\leq n}(-\psi(\bar{\beta}_{k}^{n}))\1_{[\bar{\tau}_{k}^{n}<T-\Delta]}+\1_{[\bar{\tau}_{n}^{n}<T-\Delta]} Y^0_{\bar{\tau}_n^n} (\bar{\beta}_{1}^{n}+\ldots+\bar{\beta}_{n}^{n})].
\end{array}$$
Finally, taking into account (\ref{y0ec}) and \eqref{yxi}, we get
$$\begin{array}{lll}
Y_{0}^{n}(0)&\geq&\E[\Int_{0}^{\bar{\tau}_{1}^{n}+\Delta} g(s,X_{s}) \;ds +\Sum_{1 \leq k\leq n}\Int_{(\bar{\tau}_{k}^{n}+\Delta)\wedge T}^{(\bar{\tau}_{k+1}^{n}+\Delta)\wedge T}g(s,X_{s}+\bar{\beta}_{1}^{n}+\ldots+\bar{\beta}_{k}^{n})\,ds \\
&+&\Sum_{1\leq k\leq n}(-\psi(\bar{\beta}_{k}^{n}))\1_{[\bar{\tau}_{k}^{n}<T-\Delta]}]\\
&=&J(\bar{\delta}_{n}).
\end{array}$$
Henceforth, $Y_{0}^{n}(0)=J(\delta^{n})\geq J(\bar{\delta}_{n})$ which implies that the strategy $\delta^n$ is optimal over the set $\A_n$.
\end{proof}

By using the last result and taking   account of Proposition \ref{sup}, we obtain:
\begin{Theorem}
   The strategy $\delta^\mathfrak{X}$ is optimal over the set $\A$, that is,
   $$Y_0^\mathfrak{X}(0)=J(\delta^\mathfrak{X})=\sup_{\delta \in \A_\mathfrak{X}}J(\delta)=\sup_{\delta \in \A}J(\delta).\qed$$ 
\end{Theorem}

\subsection{The risk sensitive case} \label{risk}
In this section, we suppose that the controller is sensitive with respect to risk, then he/she aims to make safe decisions by balancing both the expected yield and the possible risk.  We use an exponential function to model his/her utility which means that the total expected reward is given by:
\def \exp{\text{exp}}
\begin{equation}\lb{expcout}
    J(\delta) = \E\big[ \exp \,\theta \{\int_0^T g(s,X_s^{\delta}) ds-\sum_{n\geq 0} \psi(\xi_n) \1_{[\tau_n < T-\Delta]}\}\big]. 
\end{equation}
The role of the parameter $\theta$ is to adjust the sensitivity of the controller  with respect to risk. It can be positive in risk seeking scenario and negative when the controller is averse to risk. Without loss of generality and for the sake of simplicity, we suppose that $\theta = 1$. On the other hand on $g(.)$ and $\psi(.)$ we assume:  
\medskip
\def \cp {\mathcal{P}}

\noindent \begin{Assumption}\label{assumpt-expo}
${}$

\noindent i) The  mapping $(t,\omega,x)\in [0,T]\times\Omega\times  \R^{\ell} \mapsto g(t,\omega,x)\in \R$ is $\mathcal{P}\otimes \mathcal{B}(\R^{\ell})/\mathcal{B}(\R)$-measurable. 
Moreover there exists an $\R$-valued non-negative $\cp$-measurable process $(\gamma_t)_{t\le T}$ such that:
$$\E\big[ e^{\int_0^T \g_sds}\big]<\infty \mx{ and }\P-a.s. \mbox{ for any }(t,x)\in[0,T]\times\R^{\ell},\,|g(t,\omega,x)|\le \gamma_t(\omega).$$

\noindent ii) For any $a\in \R^{\ell}$, $\psi(a)\ge 0.$
\end{Assumption}
\subsubsection{The iterative scheme}
Our  scheme will be defined via iterative Snell envelopes in the following way:  $\forall t \leq T$, $\forall a \in \rl$,
\begin{equation}\label{y0}
    Y_t^0(a) = \E[\exp\{\int_t^T g(s,X_s+a) ds \}|\F_t].
\end{equation}
For every $ n \geq 1$,
$   {( Y_t^n(a)\; \exp\{\int_0^{t} g(s,X_s+a))}_{t \leq T}$ is the Snell envelope of the process\\ ${(\exp\{\int_0^{t} g(s,X_s+a)ds\} O_{t}^n (a))}_{t \leq T}$, that is, 
\begin{equation}
    Y_t^n(a) = \esssup_{\tau \in {\cal T}_t} \E[\exp\{\int_t^{\tau} g(s,X_s+a)ds\} O_{\tau}^n (a)|\F_t],
\end{equation}
where for any $t\le T$, 
\begin{eqnarray}\label{onexp}
    O_t^n(a)&=&\1_{[t<T-\Delta]}\E[\max_{\beta \in U} \{\exp[\int_t^{t+\Delta} g(s,X_s+a)ds-\psi(\beta) \1_{[t<T-\Delta]}]\, Y_{t+\Delta}^{n-1}(a+\beta) \}|\F_t] \nonumber\\&&\qq\qq+\1_{[T-\Delta\leq t \leq T]}\E[ \exp\{\int_t^{T} g(s,X_s+a)ds\}|\F_t].
\end{eqnarray}The process $O^n(a)$ is the predictable projection of 
\begin{eqnarray}
   \bar {\cal L}^n_t(a)=\1_{[t<T-\Delta]}\max_{\beta \in U} \{\exp[\int_t^{t+\Delta} g(s,X_s+a)ds-\psi(\beta) \1_{[t<T-\Delta]}]Y_{t+\Delta}^{n-1}(a+\beta) \}\nn\\+\1_{[T-\Delta\leq t \leq T]}\E[ \exp\{\int_t^{T} g(s,X_s+a)ds\}, \,t\le T.\nn
\end{eqnarray}
We then have:
\begin{Proposition}\label{propexp}
For any $n \in \mathbb N$ and  any $ a \in \rl$, 
    \begin{enumerate}
        \item[(i)]   The processes $Y^n(a)$ is well defined, continuous and verifies: 
        \begin{equation*}
           \forall t \leq T, \,\,0\le Y_t^n (a) \leq \E[e^{\int_t^T\gamma_sds}|\cf_t]
         ,\,\,\P-\mbox{a.s.}  \text{ and }Y_T^n(a) = 1.
        \end{equation*}
        \item[(ii)]  $\forall t \in [T-\d,T]$, $$Y_t^{n+1}(a)=O_t^{n+1}(a)=  \E[e^{\int_t^T g(s,X_s+a)ds}|\F_t].$$
    \end{enumerate}
\end{Proposition}
\begin{proof} We first focus on point (i).  
We proceed by induction.  Let us show the property for $n=0$. Let $a$ be an arbitrary element of $\rl$. Since $g(.)$ verifies 
Assumption \ref{assumpt-expo}-i), then the process $Y^0(a)$ is obviously well defined and continuous since, for any $t\le T$, 
$$
Y_t^0(a) = \underbrace{\E[\exp\{\int_0^T g(s,X_s+a) ds \}|\F_t]}_{\Theta_t}\exp\{-\int_0^t g(s,X_s+a) ds\}.
$$
But $(\Theta_t)_{t\le T}$ is a continuous martingale as the filtration $(\cf_t)_{t\le T}$ is Brownian, thus $Y^0(a)$ is continuous. 
In addition, by \eqref{y0} and Assumption \ref{assumpt-expo}, i) we have: $\forall t\le T$, 
     \begin{eqnarray*}
        0\le Y_t^0(a)\leq \E[e^{\int_t^T\gamma_sds}|\cf_t], \P-\mbox{a.s.}
    \end{eqnarray*}and finally $Y_T^0(a)=1$. Thus the property is valid for $n=0$. Assume now that the property is valid for some $n\ge 0$, i.e., for any $a\in \rl$, 
the processes $Y^n(a)$ is well defined, continuous and verifies: 
        $$
           \forall t \leq T, \,\,0\le Y_t^n (a) \leq \E[e^{\int_t^T\gamma_sds}|\cf_t],\, \P-\mbox{a.s.}  \text{ and }Y_T^n(a) = 1.
$$The process $(O_t^{n+1}(a))_{t\le T}$ is $\cadlag$. It is continuous except possibly at $T-\d$, since it may have a non negative jump at that point and it satisfies: $\forall t\le T$, 
    \begin{eqnarray*}
          0\le O_t^{n+1}(a) &=&\1_{[t<T-\Delta]}\E[\max_{\beta \in U} \{\exp(\int_t^{t+\Delta} g(s,X_s+a)ds-\psi(\beta)\1_{[t<T-\Delta]}) Y_{t+\Delta}^{n}(a+\beta) \}|\F_t]\\&\qq & \qq +\1_{[T-\Delta\leq t \leq T]}\E[ \exp\{\int_t^{T} g(s,X_s+a)ds\}|\F_t]\\
          &\leq& \1_{[t<T-\Delta]}\E[e^{\int_t^{t+\d}\g_sds}\E[e^{\int_{t+\d}^T\g_sds}|\F_{t+\d}]|\cf_t]+\1_{[T-\Delta\leq t \leq T]}\E[e^{\int_t^T\g_sds}|\cf_t]\\
         &=& \E[e^{\int_t^T\g_sds}|\cf_t].
    \end{eqnarray*}
    Next the process ${(\exp\{\int_0^{t} g(s,X_s+a)ds\} O_{t}^{n+1} (a))}_{t \leq T}$ is $\cadlag$, positive and upper bounded by the martingale $(\E[e^{\int_0^T\g_sds}|\cf_t])_{t\le T}$ then it is of class [D]. Therefore, the process 
     $Y^{n+1}(a)$ defined by 
     \begin{eqnarray*}
          Y_t^{n+1}(a) &=& \esssup_{\tau \in {\cal T}_0} \E[\exp\{\int_t^{\tau} g(s,X_s+a)ds\} O_{\tau}^{n+1} (a)|\F_t],\,\,t\le T,
    \end{eqnarray*}
    is well posed, $\cadlag$, belongs to class [D] and satisfies 
    $Y_T^{n+1}(a)=O_T^{n+1}(a)=1$. It is also continuous since the process $(O_t^{n+1}(a))_{t\le T}$ is c\`adl\`ag and has possibly a positive jump at $T-\d$ and $\exp(.)\ge 0$ (see Proposition \ref{cont-opt}). Finally for any $t \leq T$, we have:
    \begin{eqnarray*}
        Y_t^{n+1}(a) &=& \esssup_{\tau \in {\cal T}_0} \E[\exp\{\int_t^{\tau} g(s,X_s+a)ds\} O_{\tau}^{n+1} (a)|\F_t]\\
        &\leq& \E[e^{\int_t^\t\g_sds}\E[e^{\int_\t^T\g_sds}|\cf_\t]|\cf_t]=\E[e^{\int_t^T\g_sds|}\cf_t].
\end{eqnarray*}
Thus the property is valid for $n+1$ and then it is valid for any $n\ge 0$.

We now focus on point (ii).
For any $t\in [T-\d,T]$ and any $a\in \rl$,  
\begin{eqnarray*}\barl
    O_t^n(a)=\E[ \exp\{\int_t^{T} g(s,X_s+a)ds\}|\F_t] \ear
\end{eqnarray*}
and 
\begin{align*}\barl
    Y^n_t(a)=
    \esssup_{\tau \in {\cal T}_t} \E[\exp\{\int_t^{\tau} g(s,X_s+a)ds\} O_{\tau}^{n} (a)|\F_t]=\E[\exp\{\int_t^{T} g(s,X_s+a)ds\} |\F_t]=O_t^n(a).
    \ear
\end{align*}
\end{proof}

Next for  every $n \in \N$,  and $\xi$ a random variable  which takes its values in 
\def \ss {\mathcal S_c^2}
$\mathcal{V}$ (once more which is finite), let us set: 
\begin{equation}\lb{ytxiexp}
   Y_t^n(\xi) = \sum_{a \in \mathcal V}  Y_t^n(a) \xa 
\text{ and }O_t^n(\xi) = \sum_{a \in \mathcal V}  O_t^n(a) \xa,\, t\le T.
\end{equation}
As a by-product of Proposition \ref{propexp}, we have: 
\begin{Corollary} \label{corexp} Let $\t$ be a stopping time and $\xi$ an $\cf_\t$-random variable valued in $\mathcal{V}$, then for any $n\ge 0$ and $t\in [\t,T]$, on has:
$$
0\le Y_t^0(\xi) \leq Y_t^{n}(\xi)\leq \E[e^{\int_\t^{T} \g_sds}|\F_t],\, \P-a.s.
$$
Moreover if $\t\ge T-\d$, then $Y_t^{n}(\xi) = O_t^{n}(\xi)$ for any $t\in [\t,T]$ and $n\ge 1$.\qed
\end{Corollary}
\subsubsection{The optimal strategy}
Let $n\ge 1$ be fixed. We will construct a strategy ${\tilde{\delta}}^n = {(\tilde{\tau}_k^n,\tilde{\beta}_k^n)}_{k=1,...,n}$ that will be proved in the sequel, to be optimal in $\ca_n$. So first let 
$$
\tilde{\tau}_1^n = \inf\{s \geq 0,\,\,Y_s^n(0)=O_s^n(0)\}.
$$
For $l\in \ci $ let, 
\begin{eqnarray*}
\Lambda^{1,n}_l&= \bigcap_{ 1\leq j \leq p,\, j\neq l} \Big\{\exp(\int_{\tilde{\tau}_1^n}^{\tilde{\tau}_1^n+\Delta} &g(s,X_s)ds-\psi(a_l)\1_{[\tilde{\tau}_1^n<T-\Delta]}) Y_{\tilde{\tau}_1^n+\Delta}^{n-1}(a_l) \\{ } && \geq  \exp(\int_{\tilde{\tau}_1^n}^{\tilde{\tau}_1^n+\Delta} g(s,X_s)ds-\psi(a_j)\1_{[\tilde{\tau}_1^n<T-\Delta]}) Y_{\tilde{\tau}_1^n+\Delta}^{n-1}(a_j) \Big\};
\end{eqnarray*}
\def \txt{\text}
 \begin{eqnarray}\tbar{\Lambda}^{1,n}_1 =\Lambda^{1,n}_1 \txt{ and for }l\in \{ 2,\dots, p\}, \tbar{\Lambda}^{1,n}_l=\Lambda^{1,n}_l\setminus [\Lambda^{1,n}_1 \cup \ldots \Lambda^{1,n}_{l-1}], 
 \end{eqnarray}
and finally
$$
\tilde{\beta}_1^n = \sum_{1\leq l \leq p} a_l \1_{\tbar{\Lambda}^{1,n}_l}.$$
The random variable $\tilde{\beta}^n_1$ is $\F_{\tilde{\tau}_1^n + \Delta}-$measurable and verifies:
{\small{
\begin{eqnarray*}
\max_{a \in U} \{\exp(\int_{\tilde{\tau}_1^n}^{\tilde{\tau}_1^n+\Delta} g(s,X_s)ds-\psi(a)\1_{[\tilde{\tau}_1^n<T-\Delta]}) Y_{\tilde{\tau}_1^n+\Delta}^{n-1}(a) \}
 = \exp(\int_{\tilde{\tau}_1^n}^{\tilde{\tau}_1^n+\Delta} g(s,X_s)ds-\psi(\tilde{\beta}^n_1)\1_{[\tilde{\tau}_1^n<T-\Delta]}) Y_{\tilde{\tau}_1^n+\Delta}^{n-1}(\tilde{\beta}_1^n).
 \end{eqnarray*}}}
Next for $k \in \{2,...,n\}$, let $G_{k-1}^n = \{\tau_{k-1}^n \leq T-2\Delta\} \in \F_{\tau_{k-1}^n}$ and, 
 $$
 \gs_k^n = \inf\{s \geq \tilde{\tau}^n_{k-1}+\Delta,\, Y_s^{n-k}(\tilde{\beta}_1^n+\ldots+\tilde{\beta}_{k-1}^n) = O_s^{n-k} (\tilde{\beta}_1^n+\ldots+\tilde{\beta}_{k-1}^n)\},
 $$
 $$\tilde{\tau}_{k}^{n}= \bar{\gamma}_k^n \1_{G_{k-1}^n} + T \1_{\overline{G_{k-1}^n}}$$
  and finally
 $$
\tilde{\beta}_k^n = \1_{G_{k-1}^n} \big(\sum_{1\leq l \leq p} a_l \1_{\tbar{\Lambda}^{k,n}_l}\big)+a_1 \1_{\overline{G_{k-1}^n}},$$ with
\begin{eqnarray*}
\Lambda^{k,n}_l&= \bigcap_{ 1\leq j \leq p} \Big\{&\exp(\int_{\tilde{\tau}_k^n}^{\tilde{\tau}_k^n+\Delta} g(s,X_s+\tilde{\beta}_1^n+\ldots+\tilde{\beta}_{k-1}^n)ds-\psi(a_l)\1_{[\tilde{\tau}_k^n<T-\Delta]}) Y_{\tilde{\tau}_k^n+\Delta}^{n-1}(a_l) \\{ } && \geq  \exp(\int_{\tilde{\tau}_k^n}^{\tilde{\tau}_k^n+\Delta} g(s,X_s+\tilde{\beta}_1^n+\ldots+\tilde{\beta}_{k-1}^n)ds-\psi(a_j)\1_{[\tilde{\tau}_k^n<T-\Delta]}) Y_{\tilde{\tau}_1^n+\Delta}^{n-1}(a_j) \Big\};
\end{eqnarray*}
\def \txt{\text}
 \begin{equation*}
 \tbar{\Lambda}^{k,n}_1 =\Lambda^{k,n}_1 \txt{ and for }l\in \{2,...,p\}, \tbar{\Lambda}^{k,n}_l=\Lambda^{k,n}_l\setminus [\Lambda^{k,n}_1 \cup \ldots \Lambda^{k,n}_{l-1}]. 
 \end{equation*}
 \begin{Theorem}\label{theorem-optimal}
     The strategy $\tilde{\delta}^n={(\tilde{\tau}_k^n,\tilde{\beta}_k^n)}_{k= 1,...,n}$ is optimal in $\ca_n$ for the risk-sensitive impulse control problem.
 \end{Theorem}
 \begin{proof} To begin with, let us show that $\tilde{\delta}^n={(\tilde{\tau}_k^n,\tilde{\beta}_k^n)}_{k=1,...,n}$ is admissible. First note that $\tilde \t_1\le T-\d$. Next for any $k=1,...,n-1$, on $G^n_k$, $\bar \gamma^n_{k+1} \le T-\d$, then it follows from its definition that $\tilde \t^n_{k+1}=\bar\gamma^n_{k+1}\le T$. Now if $\omega \in \bar G^n_{k}$, 
$\tilde \t^n_{k+1}(\omega)=T$. Thus 
$\tilde \t^n_{k+1}(\omega)\le T$, $\P$-a.s.
Next on $G^n_k$, 
$\tilde \t^n_{k}+\d\le \tilde \t^n_{k+1}=\bar\g^n_{k+1}\le T-\d$ which implies that on $G^n_k$, 
$\tilde \t^n_{k+1}\ge \min\{\tilde \t^n_{k}+\d, T\}$ while on  $\bar G^n_k$,
$\tilde \t^n_{k+1}=T$. It follows that 
$T\ge \tilde \t^n_{k+1}\ge \min\{\tilde \t^n_{k+1}+\d, T\}$, $\P$-a.s. Finally for any $k\in \{1,\ldots,n\}$, the sets $(\tbar{\Lambda}^{k,n}_l)_{l=1,...,p}\in \F_{\tilde{\tau}_k^n + \d}$, then $\tilde{\beta}_k^n$ is $\F_{\tilde{\tau}_k^n + \d}$, thus the strategy $\dl^n$ is  admissible.

Now, since $   {( Y_t^1(0) \exp\{\int_0^{t} g(s,X_s))}_{t \leq T}$ is the Snell envelope of the process\\ ${(\exp\{\int_0^{t} g(s,X_s)ds\} O_{t}^1 (0))}_{t \leq T}$ and the latter is continuous with a possible non-negative jump at $T-\Delta$, then $\tilde{\tau}_1^n$ is an optimal stopping time after $0$ and we have:
  \begin{eqnarray}\label{y0exp}
   Y_0^n(0) &=& \E[\exp\{\int_0^{\tilde{\tau}_1^n} g(s,X_s)ds\} O_{\tilde{\tau}_1^n}^n (0)] \nonumber  \\ &=& \E[\exp\{\int_0^{\tilde{\tau}_1^n} g(s,X_s)ds\} \{\1_{[\tilde{\tau}_1^n<T-\Delta]}\E[\exp(\int_{\tilde{\tau}_1^n}^{\tilde{\tau}_1^n+\Delta} g(s,X_s)ds-\psi(\tilde{\beta}_1^n)\1_{[\tilde{\tau}_1^n<T-\Delta]}) \nonumber \\ & &Y_{\tilde{\tau}_1^n+\Delta}^{n-1}(\tilde{\beta}_1^n)|\F_{\tilde{\tau}_1^n}] + \1_{[ \tilde{\tau}_1^n=T-\Delta]}\E[\exp[\int_{\tilde{\tau}_1^n}^{T} g(s,X_s)ds]|\F_{\tilde{\tau}_1^n}]\}] \nonumber \\ &=& \E[\1_{[\tilde{\tau}_1^n<T-\Delta]}\exp\{\int_0^{\tilde{\tau}_1^n+\Delta} g(s,X_s)ds-\psi(\tilde{\beta}_1^n)\1_{[\tilde{\tau}_1^n<T-\Delta]}\} Y_{\tilde{\tau}_1^n+\Delta}^{n-1}(\tilde{\beta}_1^n) \nonumber\\ & +&  \1_{[\tilde{\tau}_1^n =T-\d]} \exp\{\int_{0}^{T} g(s,X_s)ds\}]\nonumber\\ &=& \E[ \exp\{\int_0^{\tilde{\tau}_1^n+\Delta} g(s,X_s)ds-\psi(\tilde{\beta}_1^n)\1_{[\tilde{\tau}_1^n<T-\Delta]}\} Y_{\tilde{\tau}_1^n+\Delta}^{n-1}(\tilde{\beta}_1^n) ].
 \end{eqnarray}
 Now, notice that according to Corollary \ref{corexp}, for $ t \in [T-\Delta,  T]$, 
{{ \begin{eqnarray*}
\1_{\overline{G_1^n}} Y_{\tilde{\tau}_1^n+\Delta}^{n-1}(\tilde{\beta}_1^n) &=& \1_{\overline{G_1^n}}\E[\exp\{\Int_{\tilde{\tau}_1^n + \d}^{T} g(s,X_s+\tilde{\beta}_1^n)ds\} |\F_{\tilde{\tau}_1^n + \d}].
\end{eqnarray*} }}
On the other hand,
\begin{eqnarray*}
\1_{{G_1^n}} Y_{\tilde{\tau}_1^n+\Delta}^{n-1}(\tilde{\beta}_1^n) &=& \1_{{G_1^n}}\esssup_{\tau \in {\cal T}_{\tilde{\tau}_1^n + \d}} \E[\exp\{\int_{\tilde{\tau}_1^n + \d}^{\tau} g(s,X_s+\tilde{\beta}_1^n)ds\} O_{\tau}^{n-1}(\tilde{\beta}_1^n)|\F_{\tilde{\tau}_1^n + \d}]\\&=& \1_{{G_1^n}}  \E[\exp\{\int_{\tilde{\tau}_1^n + \d}^{\bar{\gamma}_2^n} g(s,X_s+\tilde{\beta}_1^n)ds\} O_{\bar{\gamma}_2^n}^{n-1}(\tilde{\beta}_1^n)|\F_{\tilde{\tau}_1^n + \d}].
\end{eqnarray*}
But  on ${G_1^n}$, $\tilde{\tau}_2^n = \bar{\gamma}_2^n  \le T-\Delta$, then
{\small{\begin{eqnarray*}
\1_{{G_1^n}} Y_{\tilde{\tau}_1^n+\Delta}^{n-1}(\tilde{\beta}_1^n) &=&\1_{{G_1^n}}  \E[\exp\{\int_{\tilde \tau_1^n + \d}^{\tilde{\tau}_2^n} g(s,X_s+\tilde{\beta}_1^n)ds\} \Big\{\1_{[\tilde{\tau}_2^n < T-\d]}. \\& &\E[\max_{\beta \in U} \{\exp[\int_{\tilde{\tau}_2^n}^{\tilde{\tau}_2^n+\d} g(s,X_s+\tilde{\beta}_1^n)ds-\psi(\beta) \1_{[{\tilde{\tau}_2^n}<T-\Delta]}]\, Y_{{\tilde{\tau}_2^n}+\Delta}^{n-2}(\tilde{\beta}_1^n+\beta) \}|\F_{\tilde{\tau}_2^n}]\\&{}&\qq+\1_{[{\tilde{\tau}_2^n} =T-\Delta ]}\E[ \exp\{\int_{\tilde{\tau}_2^n}^{T} g(s,X_s+\tilde{\beta}_1^n)ds\}|\F_{\tilde{\tau}_2^n}]\Big\} |\F_{\tilde{\tau}_1^n + \d}].
\end{eqnarray*}}}
Since $G_1^n$ is $\F_{\tilde{\tau}_1^n +\d}$-measurable, we obtain
 \begin{eqnarray*}
\1_{{G_1^n}} Y_{\tilde{\tau}_1^n+\Delta}^{n-1}(\tilde{\beta}_1^n) &=&\E[\1_{{G_1^n\cap [{\tilde{\tau}_2^n} <T-\Delta ]}} \exp\{\int_{\tilde{\tau}_1^n + \d}^{\tilde{\tau}_2^n+\d} g(s,X_s+\tilde{\beta}_1^n)ds-\psi(\tilde{\beta}_2^n) \1_{[{\tilde{\tau}_2^n}<T-\Delta]}\}\, Y_{{\tilde{\tau}_2^n}+\Delta}^{n-2}(\tilde{\beta}_1^n+\tilde{\beta}_2^n) 
\\&+&\1_{G_1^n \cap [{\tilde{\tau}_2^n} =T-\Delta ]} \exp\{\int_{\tilde{\tau}_1^n + \d}^{T} g(s,X_s+\tilde{\beta}_1^n)ds\}|\F_{\tilde{\tau}_1^n + \d}]
\end{eqnarray*}
Therefore
{\small{\begin{eqnarray*}
 Y_{\tilde{\tau}_1^n+\Delta}^{n-1}(\tilde{\beta}_1^n) &=& \E[ \1_{{G_1^n}\cap [{\tilde{\tau}_2^n} <T-\Delta ]} \exp\{\int_{\tilde{\tau}_1^n + \d}^{\tilde{\tau}_2^n+\d} g(s,X_s+\tilde{\beta}_1^n)ds-\psi(\tilde{\beta}_2^n) \1_{[{\tilde{\tau}_2^n}<T-\Delta]}\}\, Y_{{\tilde{\tau}_2^n}+\Delta}^{n-2}(\tilde{\beta}_1^n+\tilde{\beta}_2^n) 
\\&+&\1_{G_1^n \cap [{\tilde{\tau}_2^n} =T-\Delta ]} \exp\{\int_{\tilde{\tau}_1^n + \d}^{T} g(s,X_s+\tilde{\beta}_1^n)ds\} + \1_{\bar{G}_1^n} \exp\{\Int_{\tilde \tau_1^n + \d}^{T} g(s,X_s+\tilde{\beta}_1^n)ds\} |\F_{\tilde{\tau}_1^n + \d}].
\end{eqnarray*}}}
Replace now $Y_{\tilde{\tau}_1^n+\Delta}^{n-1}(\tilde{\beta}_1^n)$ with its last expression in (\ref{y0exp}) to obtain
 \def \og{\overline{G}}
{\small{ \begin{eqnarray*}
    Y_0^n(0) &=&    \E[   \1_{{G_1^n}\cap [{\tilde{\tau}_2^n} <T-\Delta ]} \exp\{ \int_0^{\tilde{\tau}_1^n+\Delta} g(s,X_s)ds +\int_{\tilde{\tau}_1^n + \d}^{\tilde{\tau}_2^n+\d} g(s,X_s+\tilde{\beta}_1^n)ds-\psi(\tilde{\beta}_1^n)\1_{[\tilde{\tau}_1^n<T-\Delta]}\\&-&\psi(\tilde{\beta}_2^n) \1_{[{\tilde{\tau}_2^n}<T-\Delta]}\} Y_{{\tilde{\tau}_2^n}+\Delta}^{n-2}(\tilde{\beta}_1^n+\tilde{\beta}_2^n) 
+\1_{G_1^n \cap [{\tilde{\tau}_2^n} =T-\Delta ]} \exp\{\int_0^{\tilde{\tau}_1^n+\Delta} g(s,X_s)ds \\&+&\int_{\tilde{\tau}_1^n + \d}^{T} g(s,X_s+\tilde{\beta}_1^n)ds-\psi(\tilde{\beta}_1^n)\1_{[\tilde{\tau}_1^n<T-\Delta]}\}\\& +& \1_{\overline{G_1^n}} \exp\{\int_0^{\tilde{\tau}_1^n+\Delta} g(s,X_s)ds +\Int_{\tilde{\tau}_1^n + \d}^{T} g(s,X_s+\tilde{\beta}_1^n)ds-\psi(\tilde{\beta}_1^n)\1_{[\tilde{\tau}_1^n<T-\Delta]}\}   ]\\
&=& \E[   \1_{{G_1^n}} \exp\{ \int_0^{\tilde{\tau}_2^n+\d} g(s,X_s^{\tilde\delta^n})ds-\psi(\tilde{\beta}_1^n)\1_{[\tilde{\tau}_1^n<T-\Delta]}-\psi(\tilde{\beta}_2^n) \1_{[{\tilde{\tau}_2^n}<T-\Delta]}\} Y_{{\tilde{\tau}_2^n}+\Delta}^{n-2}(\tilde{\beta}_1^n+\tilde{\beta}_2^n) 
\\ &+& \1_{\overline{G_1^n}} \exp\{\int_0^T g(s,X_s^{\tilde \delta^n})ds-\psi(\tilde{\beta}_1^n)\1_{[\tilde{\tau}_1^n<T-\Delta]}\}]
 \end{eqnarray*}}}
 By iterating the same reasoning as many times as necessary, we obtain
 \begin{eqnarray}
   Y_0^n(0) &=&  \E[   \1_{{G_{n-1}^n}} \exp\{ \int_0^{\tilde{\tau}_n^n+\d} g(s,X_s^{\tilde\delta^n})ds-\sum_{k=1}^n \psi(\tilde{\beta}_k^n)\1_{[\tilde{\tau}_k^n<T-\Delta]}\} Y_{{\tilde{\tau}_n^n}+\Delta}^{0}(\tilde{\beta}_1^n+\ldots+\tilde{\beta}_n^n) \nonumber \\
   &+& \sum_{k=1}^{n-2} \1_{{G_{k}^n}\cap \overline{G_{k+1}^n}} \exp\{ \int_0^{T} g(s,X_s^{\tilde\delta^n})ds-\sum_{i=1}^{k+1} \psi(\tilde{\beta}_i^n)\1_{[\tilde{\tau}_i^n<T-\Delta]}\}
 \nonumber \\ &+& \1_{\overline{G_1^n}} \exp\{\int_0^{T} g(s,X_s^{\tilde\delta^n})ds-\psi(\tilde{\beta}_1^n)\1_{[\tilde{\tau}_1^n<T-\Delta]}\}   ].\label{y0nrisk}
 \end{eqnarray}
Next, let us set $\tilde{\beta}^n=\tilde{\beta}_1^n+\ldots +\tilde{\beta}_{n}^n$. By virtue of (\ref{ytxiexp}) and (\ref{y0}) and since $\tilde{\beta}^n$ is $\F_{\tilde{\tau}^n_{n}+\Delta}-$measurable, we have
\begin{eqnarray*} 
Y_{\tilde{\tau}_{n}^n+\Delta}^{0}(\tilde{\beta}^n) &=& \sum _{a \in U} Y^0_{\tilde{\tau}_{n}^n+\Delta}(\tilde{\beta}_1^n+\ldots \tilde{\beta}_{n-1}^n+a) \1_{[\tilde{\beta}_{n}^n=a]}\\&=&    \sum _{a \in U} \E[\exp\{\int_{\tilde{\tau}_{n}^n+\Delta}^T g(X_s + \tilde{\beta}_1^n+\ldots \tilde{\beta}_{n-1}^n+a)ds\}|\F_{\tilde{\tau}_{n}^n+\Delta}] \1_{[\tilde{\beta}_{n}^n=a]}  \\&=& \sum _{a \in U} \E[\1_{[\tilde{\beta}_{n-1}^n=a]} \exp\{\int_{\tilde{\tau}_{n}^n+\Delta}^T g(X_s + \tilde{\beta}_1^n+\ldots \tilde{\beta}_{n-1}^n+a)ds\}|\F_{\tilde{\tau}_{n}^n+\Delta}]\\&=&
 \E[ \exp\{\int_{\tilde{\tau}_{n}^n+\Delta}^T g(X_s + \tilde{\beta}^n)ds\}|\F_{\tilde{\tau}_{n}^n+\Delta}].
\end{eqnarray*}
The last inequality plugged in (\ref{y0nrisk}) yields
$$
Y_0^n(0)=\E[ \exp \{\int_0^T g(s,X_s^{\tbar{\delta}^n}) ds-\sum_{k\leq n} \psi(\tilde{\beta}_k^n) \1_{[\ts_k^n < T-\Delta]}\}] = J(\tilde{\delta^n}). 
$$
Now, let $\bar{\delta}' = {(\bt_k,\bb_k)}_{k=1,...,n}$ be an arbitrary strategy of $\A_n$. We then have: 
\begin{eqnarray*}
 Y_0^n(0) &=& \esssup_{\tau \in {\cal T}_0} \E[\exp\{\int_0^{\tau} g(s,X_s)ds\} O_{\tau}^n (0)]\\&\geq& \E[\exp\{\int_0^{\bt_1} g(s,X_s)ds\} O_{\bt_1}^n (0)].
\end{eqnarray*}
But
\begin{eqnarray*}
     O_{\bt_1}^n (0) &=& \1_{[\bt_1<T-\Delta]}\E[\max_{\beta \in U} \{\exp(\int_{\bt_1}^{\bt_1+\Delta} g(s,X_s)ds-\psi(\beta) \1_{[\bt_1<T-\Delta]}) Y_{\bt_1+\Delta}^{n-1}(\beta) \}|\F_{\bt_1}]\\&+&\1_{[T-\Delta\leq \bt_1 ]}\E[ \exp\{\int_{\bt_1}^{T} g(s,X_s)ds\}|\F_{\bt_1}]
     \\ &\geq& \1_{[\bt_1<T-\Delta]}\E[ \{\exp(\int_{\bt_1}^{\bt_1+\Delta} g(s,X_s)ds-\psi(\bb_1) \1_{[\bt_1<T-\Delta]}) Y_{\bt_1+\Delta}^{n-1}(\bb_1) \}|\F_{\bt_1}]\\&+&\1_{[T-\Delta\leq \bt_1 ]}\E[ \exp\{\int_{\bt_1}^{T} g(s,X_s)ds\}|\F_{\bt_1}].
\end{eqnarray*}
It follows that
{\small{\begin{eqnarray*}
 Y_0^n(0) &\geq& \E[\exp\{\int_0^{\bt_1} g(s,X_s)ds\} \{\1_{[\bt_1<T-\Delta]}\E[ \{\exp(\int_{\bt_1}^{\bt_1+\Delta} g(s,X_s)ds-\psi(\bb_1) \1_{[\bt_1<T-\Delta]}) Y_{\bt_1+\Delta}^{n-1}(\bb_1) \}|\F_{\bt_1}]\\&+&\1_{[T-\Delta\leq \bt_1 ]}\E[ \exp\{\int_{\bt_1}^{T} g(s,X_s)ds\}|\F_{\bt_1}]\}]\\
 &\geq& \E[\1_{[\bt_1<T-\Delta]} \exp\{\int_0^{\bt_1+\Delta} g(s,X_s)ds-\psi(\bb_1) \1_{[\bt_1<T-\Delta]})\} Y_{\bt_1+\Delta}^{n-1}(\bb_1) +\1_{[T-\Delta\leq \bt_1]}\exp\{\int_{0}^{T} g(s,X_s)ds\}]\\
 &\geq& \E[ \exp\{\int_0^{(\bt_1+\Delta)\wedge T} g(s,X_s)ds-\psi(\bb_1) \1_{[\bt_1<T-\Delta]})\} Y_{(\bt_1+\Delta)\wedge T}^{n-1}(\bb_1) ]
\end{eqnarray*}}}since $Y_{T}^{n-1}(\bb_1)=1$.
By reasoning many times in the same way, we obtain:
\begin{eqnarray*}
 Y_0^n(0)  &\geq& \E[\1_{[\bt_{n}<T-\Delta]} \exp\{\int_0^{(\bt_{n}+\Delta) \wedge T} g(s,X_s^{\bar{{\delta'}}})ds-\sum _{k=1}^{n}\psi(\bb_k)\1_{[\bt_k<T-\Delta]})\} Y_{(\bt_{n}+\Delta)\wedge T}^{0}(\bb_1+\ldots +\bb_{n}).
\end{eqnarray*}
Finally,  using (\ref{y0}) to get
\begin{eqnarray*}
 Y_0^n(0)  \geq \E[ \exp \{\int_0^T g(s,X_s^{\bar{\delta'}}) ds-\sum_{k\leq n} \psi(\bb_k) \1_{[\bt_k < T-\Delta]}\}] = J(\bar{\delta'})
\end{eqnarray*}and then the strategy 
$\tilde \dl^n$ is optimal in $\ca_n$. 
  \end{proof}
  As a by-product, in taking $n=\mathfrak{X}=[\frac{T}{\d}]$, we obtain:
\begin{Corollary}
The risk-sensitive impulse control problem with delay $\d$ \eqref{expcout} has an optimal strategy.    
\end{Corollary}
\section{The infinite horizon case}
In this section $B=(B_t)_{t\ge 0}$ is a $d$-dimensional Brownian motion on $(\Omega,\mathcal{F},\P)$, $(\Fc_{t}^{0}:= \sigma\{B_{s}, s\leq t\})_{t\geq 0}$ is the natural filtration of $B$, $(\mathcal{F}_{t})_{t\geq 0}$ its completion with the $\P$-null sets of $\Fc$ and finally $\cf_\infty= \bigvee_{t\ge 0}\cf_t$. On the other hand, 
\begin{itemize}
\item $\mathcal{P}$ is the $\sigma$-algebra on $[0,+\infty[\times\Omega$
 of $(\mathcal{F}_{t})_{t\ge 0}$-progressively measurable processes;
\item [$\bullet$] $\mathcal{H}^{p,n}_\infty=\{(v_{t})_{ t\geq 0}:$ $\mathcal{P}$-measurable, $\R^{n}$-valued process s.t.
$\E[\int_{0}^{\infty}|v_{s}|^{p}\,ds]<\infty$\}\,($n\ge 1$ \mx{and} $p>1$);
\item [$\bullet$] $\mathcal{S}^{2}_\infty=\{(Y_t)_{0\le t\le \infty}$, $\mathcal{P}$-measurable  c\`{a}dl\`{a}g (or rcll) process, s.t.
$\E[\sup_{0\le t\le \infty}|Y_{t}|^{2}]<\infty\}$;

\item [$\bullet$] $\mathcal{S}_{c,\infty}^{2}$ (resp. 
$\mathcal{S}_{i,\infty}^{2}$) is the subset of 
$\mathcal{S}^{2}_\infty$ of continuous processes (resp. continuous non-decreasing 
processes $(k_{t})_{0\le t\le \infty}$ such that $k_0=0$);

\item [$\bullet$] $L^{2}_\infty=\{\eta :  \mathcal{F}_{\infty}-$ measurable random variable, s.t.
$\E[|\eta|^{^{2}}]<\infty\}$;

\item [$\bullet$]$\mathcal{T}_{t_0}=\{\nu: (\mathcal{F}_{t})_{t\geq 0}$-stopping time, s.t. $\P$-a.s, $t_0\leq \nu \}$ $(t_0\ge 0)$.
\end{itemize}

Next, for any stochastic process $y:=(y_t)_{t\in [0,\infty)}$,  we  use the following notation:  $y_{\infty}=\limsup_{t\rightarrow \infty}y_t$.
In addition, the process $y$ is said to be continuous at $+\infty$ if $\lim_{t\rightarrow \infty}y_t$ exists.  Finally, if $y$ 
is a {c\`adl\`ag}  $\mathcal{F}_t$-supermartingale, which is non-negative or bounded from below, then according to (\cite{Karatzas3}, pp.18), it is continuous at $+\infty$.

In this section we are considering the impulse control problem in infinite horizon. An impulse strategy 
$\dl$ is ${(\tau_n,\xi_n)}_{n\ge 1}$ where for any $n\ge 1$, $\t_n$ is an $\cft$-stopping time such that $\tau_n+\Delta \leq \tau_{n+1}$ and 
$\xi_n$ a random variable with values in $U$ a finite subset of $\rl$ and $\cf_{\t_n+\d}$-measurable. We denote by $\A_\infty$ the set of admissible strategies.

When a strategy $\dl$ is implemented by the controller his/her payoff is given by: 
\begin{equation*}
J(\delta):=\E[\Phi\{\int_{0}^{+\infty} e^{-rs}g(s,X_{s}^{\delta})\;ds-\sum_{n\geq
1}e^{-r\tau_n}\psi(\xi_{n})\1_{[\tau_{n}< {+\infty}]}\}], 
\end{equation*}
with: 

i) the controlled process $X^\dl$, taking its values in $\R^\ell$, is given by: 
$$\forall t\ge 0,\,X_{t}^{\delta} =X_{t}+\sum_{n\geq 1}\xi_{n}\1_{[\tau_{n}+\Delta \leq t]} $$
where $(X_t)_{t\ge 0}$ is a $\mathcal P$-measurable process with values in $\R^\ell$ which stands for the dynamics of the uncontrolled system.

ii) $\Phi(.)$ is the utility function and $r$ is a discount factor.
\ms

\noindent 
We aim to find an  optimal strategy $\delta^{*}$, i.e., satisfying: $$
\sup_{\delta \in \A_\infty} J(\delta) = J(\delta^{*}).
$$
An example of the stochastic process above we have in mind is:
$$\barl
\forall t\ge 0,\, X_t=x+\int_0^tb(s,\omega)ds+\int_0^t\sigma(s,\omega)dB_s\ear$$ where $(b_t)_{t\ge 0}$ (resp. $(\sigma_t)_{t\ge 0}$) is a process of 
$\mathcal{H}^{1,\ell}_\infty$ (resp. $\mathcal{H}^{2,\ell \times d}_\infty$). For general processes $b$ and $\sigma$, the It\^o process $X$ is no longer 
a Markov process and then the deterministic methods cannot be applied to deal with this problem. 
\ms

We are going to consider the risk-neutral case, i.e., $\Phi(x)=x$. The risk-sensitive case can be considered in the same way. On the other hand, we assume that the functions $g(.)$ and $\psi(.)$ verify:
\begin{Assumption}\label{assumpt3}
${}$

\noindent i) The  mapping $(t,\omega,x)\in [0,\infty[\times\Omega\times  \rl \mapsto g(t,\omega,x)\in \R$ is $\mathcal{P}\otimes \mathcal{B}(\rl)/\mathcal{B}(\R)$-measurable. 
Moreover there exists a non-negative process $(\gamma_t)_{t\ge 0}$ such that:$$\P-a.s. \mbox{ for any }(t,x)\in[0,+\infty[\times\rl,\,|g(t,\omega,x)|\le \gamma_t(\omega) \mx{ and }\E[(\int_0^\infty e^{-rs}\g_sds)^2]<\infty.$$

\noindent ii) For any $a\in \R^\ell$, $\psi(a)\ge 0
.$

\end{Assumption}
\subsection{Construction of the iterative scheme}

 In the sequel, $a$ will stand for an element of $\rl$. So let us consider the following scheme: 
 
The pair of processes $(Y^{0}(a ),Z^{0}(a))$ is a solution of the following standard BSDE in infinite horizon which exists according to the result by Chen \cite{Chen}, pp.96: 
\be\left\{
\barl
(Y^{0}(a ),Z^{0}(a))\in \mathcal{S}_{c,\infty}^{2}\times \mathcal{H}^{2,d}_\infty;\\\\
Y_{t}^{0}(a)=\int_{t}^{\infty}e^{-rs} g(s,X_{s}+a)\,ds-\int_{t}^{\infty}Z_s^{0}(a)dB_s, \,t\in [0,+\infty].
\ear\right.
\ee 
First let us point out that $Y_{\infty}^{0}(a)=0$. Next for any $t\ge 0$ we have, 
\begin{equation}\barl Y_{t}^{0}(a )=\E[\int_{t}^{\infty}e^{-rs} g(s,X_{s}+a)\,ds|\cf_t]
\label{eq4}
\ear
\end{equation} 
which implies that for any 
$a\in \rl$, and for any 
$t\ge 0$, 
\begin{equation}|Y_{t}^{0}(a )|\le \E[\int_{t}^{\infty}e^{-rs} \g_sds|\cf_t].
\label{eq4x}
\end{equation}
Now for $t\in [0,+\infty[$, let  $M_t := \mathbb E[\int_0^{\infty} e^{-rs} \gamma_s\,ds \mid \mathcal F_t].$
Then by Assumption \ref{assumpt3}-i), $(M_t)_{t\ge0}$ is a uniformly integrable martingale. Next using the martingale convergence theorem, there exists
$M_\infty \in L^1(d\P)$ such that $M_t \longrightarrow M_\infty
 \text{ a.s. and in } L^2(d\P).$ 
Since $\int_0^{\infty} e^{-rs} \gamma_s\,ds$ is $\mathcal F_\infty$-measurable, 
we have $M_\infty = \int_0^{\infty} e^{-rs} \gamma_s\,ds$  (see \cite{Revuz}, Chapter 2, Theorem 3.1 and Corollary 2.4). Consequently,
\be\lb{limm}
\lim_{t\to\infty} \mathbb E\!\left[\int_t^{\infty} e^{-rs} \gamma_s\,ds \,\middle|\, \mathcal F_t\right]
=
\lim_{t\to\infty} \bigl(M_t - \textstyle\int_0^t e^{-rs} \gamma_s\,ds\bigr)
= 0
\quad \text{$\P$-a.s. and in } L^2(d\P).
\ee
As $g(.)$ verifies Assumption \ref{assumpt3}-i), then from \eqref{eq4x} we deduce that $\lim_{t\rw \infty}|Y^0_t(a)|=0$ uniformly w.r.t $a\in \rl$. 

Next for any $n \geq1$, let $(Y^{n}(a),Z^{n}(a),K^{n}(a))$ be a triple of processes  which satisfies the following reflected BSDE in infinite horizon:  
\begin{equation}\lb{eq5x}
\left\{
  \begin{array}{ll}
  \,\,&(Y^{n}(a),Z^{n}(a),K^{n}(a)) \in \mathcal{S}_{c,\infty}^{2} \times \Hc^{2,d}_\infty \times {\mathcal{S}_{i,\infty}^{2}};\\\\
     \,\,& Y_{t}^{n}(a)=\int_{t}^{\infty}e^{-rs} g(s,X_{s}+a) \,ds+K_{\infty}^{n}(a)-K_{t}^{n}(a)-\int_{t}^{\infty}Z_{s}^{n}(a)\,dB_{s},\, t\in [0,+\infty];\\\\
     \,\, & Y_{t}^{n}(a)  \geq  O_{t}^{n}(a):=  \E  \big\lbrack\int_{t}^{t+\Delta}e^{-rs} g(s,X_{s}+a) \,ds |\Fc_{t} \big\rbrack+ \\&\qq\qq\qq\qq\qq\qq\E\big\lbrack \max\limits_{\beta\in U} \big\lbrace-e^{-r(t+\Delta)}\psi(\beta)+Y_{t+\Delta}^{n-1}(a+\beta) \big\rbrace|\Fc_{t} \big\rbrack,\,\, t\in [0,+\infty[;\\
  \,\, & \int_{0}^{\infty}(Y_{t}^{n}(a)-O_{t}^{n}(a))\,dK_{t}^{n}(a)=0.
  \end{array}  
  \right.
\end{equation}
\def \ss{\mathcal{S}_{c,\infty}^2}
For any $n\ge 1$, the process $(O_{t}^{n}(a))_{t\ge 0}$ is defined as the $\cft$-predictable projection of 
$$
{\cal L}_{t}^{n}(a):=\int_{t}^{t+\Delta}e^{-rs} g(s,X_{s}+a)ds+
\max\limits_{\beta\in U} \big\lbrace-e^{-r(t+\Delta)}\psi(\beta)+Y_{t+\Delta}^{n-1}(a+\beta) \big\rbrace,\,\,t\ge 0. 
$$
As mentioned in the very beginning, its value for $t=+\infty$ is $\lim_{t\rw +\infty}O_{t}^{n}(a)$ which we show that it exists and equal to 0.

\begin{Proposition} \label{Proposition1} For any $n\ge 1$, for any $a\in \rl$, we have:

a) The process $(O_{t}^{n}(a))_{t\ge 0}$ is continuous, $\E[\sup_{t\ge 0}|O_{t}^{n}(a)|^2]<\infty$ and $O_{\infty}^{n}(a):=\lim_{t\rw +\infty}O_{t}^{n}(a)~=~0$.

b)  The triple of processes $(Y^{n}(a),Z^{n}(a),K^{n}(a))$ is well-posed.

c) For any $ t\in [0,+\infty]$,
 \begin{equation}
Y_{t}^{n}(a)=\esssup_{\tau \in
\mathcal{T}_t}\E \big\lbrack\int_{t}^{\tau} e^{-rs}g(s,X_{s}+a)\,ds+  O_{\tau}^{n}(a)|\mathcal{F}_{t} \big\rbrack.
\label{infordrenn1}\end{equation}

d) For any $ t\geq 0$,
 \begin{equation}
Y_{t}^{0}(a)\le Y_{t}^{n}(a) \leq \E[\int_t^\infty e^{-rs}{\gamma}_{s} ds|\Fc_t].
\label{infordrenn2}\end{equation}

e) For any $a\in \rl$, $Y^{n}(a)\le Y^{n+1}(a).$
\end{Proposition}
\begin{proof} We first deal with points a), b), c) and d). We use induction and first we focus on the case when $n=1$.

As mentioned previously, for any $a\in \rl$, the process $Y^{0}(a)$ exists and belongs to 
$\mathcal{S}_{c,\infty}^2$. As $U$ is finite then the process 
$({\cal L}^1_t(a))_{t\ge 0}$ is continuous and consequently $(O_{t}^{1}(a))_{t\ge 0}$ is so (see Theorem \ref{thmcontproj} in Appendix). On the other hand for any $t\ge 0$, 
\begin{align*}
|O_{t}^{1}(a)|&\le \E[\int_{t}^{t+\Delta}e^{-rs} |g(s,X_{s}+a)|ds  +Ce^{-r(t+\Delta)}+\sum_{\b \in U}\sup_{t\ge 0}|Y_t^{0}(a+\beta)||\cf_t]\\
&\le \E[\int_{0}^{\infty}e^{-rs} \g_s ds  +Ce^{-r(t+\Delta)}+\sum_{\b \in U}\sup_{t\ge 0}|Y_t^{0}(a+\beta)||\cf_t]
\end{align*}
Using now Doob's maximal inequality to deduce
\begin{align*}
\E[\sup_{t\ge 0}|O_{t}^{1}(a)|^2]\le C(1+\E[\{\int_{0}^{\infty}e^{-rs} \g_sds\}^2+\sum_{\b \in U}\sup_{t\ge 0}|Y_t^{0}(a+\beta)|^2]<\infty.
\end{align*}
Note that the constant $C$ may change from line to line. Next, once again since $U$ is finite and using the inequality \eqref{eq4x}, we have: 
\begin{align*}
|O_{t}^{1}(a)|&\le \E  \big\lbrack\Int_{t}^{t+\Delta}e^{-rs} \g_sds +Ce^{-r(t+\Delta)}+\E[\int_{t+\d}^\infty e^{-rs}{\gamma}_{s} ds|\Fc_{t+\d}]|\Fc_{t} \big\rbrack\\&=\E  \big\lbrack\Int_{t}^{\infty}e^{-rs} \g_sds +Ce^{-r(t+\Delta)}|\Fc_{t} \big\rbrack\rw 0 \mx{ as }t\rw +\infty\end{align*}
and then $\lim_{t\rw \infty}O^1_t(a)=0=O^1_\infty(a)$ for any $a\in \rl$. Therefore by Theorem 3.2 in \cite{ham-lep}, the triple $(Y^{1}(a),Z^{1}(a),K^{1}(a))$ satisying \eqref{eq5x} (with $n=1$) exists and is unique. Additionally $Y^1(a)$ verifies: 
\begin{equation}\forall t\in [0,\infty], \,\,
Y_{t}^{1}(a)=\esssup_{\tau \in
\mathcal{T}_t}\E \big\lbrack\int_{t}^{\tau} e^{-rs}g(s,X_{s}+a)\,ds+  O_{\tau}^{1}(a)|\mathcal{F}_{t} \big\rbrack \,\,(Y_{\infty}^{1}(a)=0).
\label{infordrenn1xx}\end{equation}
Finally taking $\t=+\infty$ in \eqref{infordrenn1xx} to obtain:  For any $t\ge 0$, 
\begin{equation}
Y_{t}^{1}(a)\ge \E \big\lbrack\int_{t}^{+\infty} e^{-rs}g(s,X_{s}+a)\,ds|\mathcal{F}_{t}\big\rbrack=Y_{t}^{0}(a)
\label{infordrenn1x}\end{equation}
since, once more for any $a\in \rl$, $O^1_\infty(a)=0.$ On the other hand, for any $t\in [0,\infty[$, $$
O_{t}^{1}(a):=  \E  \big\lbrack\int_{t}^{t+\Delta}e^{-rs} g(s,X_{s}+a) \,ds  +\max\limits_{\beta\in U} \big\lbrace-e^{-r(t+\Delta)}\psi(\beta)+Y_{t+\Delta}^{0}(a+\beta) \big\rbrace|\Fc_{t} \big\rbrack. 
$$
But $\psi(.)\ge 0$, $Y^0(a)$ verifies \eqref{eq4x} and finally using Assumption \ref{assumpt3}-i) to obtain:  $\forall t\ge 0$, 
\be
O_{t}^{1}(a)\le \E  \big\lbrack\int_{t}^{t+\Delta}e^{-rs} \g_sds  +\E[\int_{t+\d}^\infty e^{-rs} \g_sds|\cf_{t+\d} ]|\Fc_{t} \big\rbrack=\E  \big\lbrack\int_{t}^{\infty}e^{-rs} \g_sds  |\Fc_{t} \big\rbrack.
\lb{eq4xxx}\ee
Going back to \eqref{infordrenn1xx} and using the inequality \eqref{eq4xxx} to obtain: $\frto$,
\begin{equation}
Y_{t}^{1}(a)\le \esssup_{\tau \in
\mathcal{T}_t}\E \big\lbrack\int_{t}^{\tau} e^{-rs}\g_sds+  \E  \big\lbrack\int_{\t}^{\infty}e^{-rs} \g_sds  |\Fc_{\t} \big\rbrack\mathcal{F}_{t} \big\rbrack =
\E \big\lbrack\int_{t}^{\infty} e^{-rs}\g_sds|\mathcal{F}_{t} \big\rbrack.
\label{infordrenn2x}\end{equation}
Thus for $n=1$, a), b), c) and d) are satisfied. 
Next assume that for some $n\ge 1$, a), b), c) and d) are satisfied. As for $O^1(a)$, the process $O^{n+1}(a)$ is continuous, uniformly $d\P$-square integrable and $\lim_{t\rw \infty}O^{n+1}_t(a)=0=O^{n+1}_\infty(a)$. Thus, once more by Theorem 3.2 in \cite{ham-lep}, for any $a\in \rl$, the triple 
$(Y^{n+1}(a),Z^{n+1}(a),K^{n+1}(a))$ verifying \eqref{eq4} exists and is unique. Moreover the process $Y^{n+1}(a)$ verifies \eqref{infordrenn1} with $O^{n+1}(a)$. Now since 
$Y^n(a)$ verifies \eqref{infordrenn2} for any $a\in \rl$ and $\psi(.)\ge 0$, then 
$O^{n+1}(a)$ verifies the inequality \eqref{eq4xxx} above and then for any $a\in \rl$, $Y^{n+1}(a)$ verifies \eqref{infordrenn2x}. The fact that $Y^{n+1}(a)\ge Y^{0}(a)$ is easily deduced from \eqref{infordrenn1} in taking $\t=+\infty$. Thus 
$Y^{n+1}(a)$ verifies \eqref{infordrenn2}.

As a result, points a), b), c) and d) are verified for $n+1$ and then they are valid for any $n\ge 1$.  

Let us now focus on point e). For $n=1$, this is just \eqref{infordrenn1x}. To conclude it is enough to use an induction argument and once more the representation \eqref{infordrenn1}. 
\end{proof}

Next we will focus on the limit of $Y^n(a)$ and we have the following:
\begin{Proposition}\label{limyn}
  For any $a \in \rl$, the sequence ${(Y^n(a))}_{n\ge 0}$ converges to a   process $Y(a) \in \mathcal{S}_\infty^2$ satisfying:
  
  i) $\forall t \in  [0,+\infty[$, $Y^0_t(a)\le Y_t(a)  \leq  \E[\int_t^\infty {\gamma}_{s} e^{-rs} ds|\Fc_t].$ 

ii) $\forall t \in  [0,+\infty]$,
      \be \lb{eqlim}
      Y_t(a) = \esssup_{\tau \in {\cal T}_t} \E[\int_t^\tau e^{-rs} \,g(s,X_s+a) ds + O_\tau (a)|\Fc_t],
      \ee
     where $(O_t(a))_{t\ge 0}$ is $\cadlag$ and the pointwise limit  of $(O_t^n(a))_{t\ge 0}$, given by: $\forall t\in [0,\infty[$, 
      \begin{equation}\label{infO}
     O_{t}(a):=  \E  \big\lbrack\Int_{t}^{t+\Delta}e^{-rs} g(s,X_{s}+a) \,ds |\Fc_{t} \big\rbrack+ \E\big\lbrack \max\limits_{\beta\in U} \big\lbrace-e^{-r(t+\Delta)}\psi(\beta)+Y_{t+\Delta}(a+\beta) \big\rbrace|\Fc_{t} \big\rbrack
      \end{equation}and $O_{\infty}(a)=0$.
  \end{Proposition}
\begin{proof}  
The sequence $(Y^n(a))_{n\ge 0}$ is increasing and taking into account of \eqref{infordrenn2}, the limit $Y_t(a):=\lim_{n\rw +\infty}Y^n_t(a)$, $t\in [0,\infty]$, exists. Next the property (i) is obviously obtained by taking the limit in (\ref{infordrenn2}).
Now, taking into account  (\ref{infordrenn1}), the process \\${(Y_t^n(a)+\int_0^t g(X_s+a) ds )}_{t\in [0,+\infty]}$ can be seen as the Snell envelope, and then an \\$(\cf_t)_{t\in [0,+\infty]}$-supermatingale, of the continuous process ${(\int_0^t g(X_s+a) ds )+O_t^n(a))}_{t\in [0,+\infty]}$ which implies that its increasing limit ${(Y_t(a)+\int_0^t g(X_s+a) ds )}_{t\in [0,+\infty]}$  is c\`adl\`ag in $[0,+\infty]$ (see e.g. \cite{Dellacherie}, pp. 86). Therefore the process $(Y_t(a))_{t\in [0,+\infty]}$ is c\`adl\`ag for any $a\in \rl$. It follows that, since $U$ is finite, the process $(\max\limits_{\beta\in U} \big\lbrace-e^{-r(t+\Delta)}\psi(\beta)+Y_{t+\Delta}(a+\beta)\big\rbrace)_{t\in [0,\infty]}$ is also $\cadlag$ and then its $\cft$-predictable projection is so. Thus the process $(O_t(a))_{t\ge 0}$ is $\cadlag$ and verifies the first equation of \eqref{infO} by taking the limit w.r.t $n$ in each hand-sides of the definition of $O^n(a)$ in \eqref{eq5x}. Finally note that \eqref{infO} and point i) imply that $O_\infty(a):=\lim_{t\rw \infty}O_t(a)=0$. Thus  for any $a\in \rl$, for any $t\in [0,+\infty]$, $\lim_{n\rw \infty}\nearrow O^n_t(a)=O_t(a)$ pointwise. Using now Proposition \ref{increaseconv} in Appendix and taking the limit in both hand-sides of \eqref{infordrenn1} one deduces that 
for any $a\in \rl$, \eqref{eqlim} is satisfied. 
\end{proof}

We now focus on the continuity of the process $Y(a)$. So for a stopping time $\t$ and $\xi$ a random variable $\cf_\t$-measurable which takes its values in a finite set $\cu\subset \rl$ we define $(\Pi_t(\xi))_{t\ge \t}$ by:
\be\lb{eqpi}
\forall t\ge \t,\,\,\Pi_t(\xi) = \sum_{a \in \cu} \Pi_t(a) \1_{[\xi=a]},
\ee
with $\Pi(a)$ is one of the processes $Y^n(a),\,Y(a),\, Z^n(a),\, K^n(a)$.
\begin{Proposition}\label{Proposition_cont} For any $a\in \rl$, the  processes $(Y_{t}(a))_{t\in [0,\infty]}$ and $(O_{t}(a))_{t\in [0,\infty]}$ are continuous.
\end{Proposition}
\begin{proof}
We will show that the process $(Y_{t}(a))_{t\in [0,\infty]}$ is continuous and then continuity of $(O_{t}(a))_{t\in [0,\infty]}$ will follow by Theorem \ref{thmcontproj} in Appendix. First recall that for any $a\in \rl$, $O_{\infty}(a)=\lim_{t\rw \infty}O_{t}(a)=0.$

Next for any predictable stopping time $\tau $ and any finite random variable $\eta$ which is $\cf_\t$-measurable with values in $\rl$, let us set $A_\tau(\eta):=\{ Y_\tau(\eta) - Y_{\tau-}(\eta)<0\}$.

So, let $T$ be a predictable stopping time for which $\mathbb P[A_T(a)]>0$. According to Proposition \ref{snellenv} in Appendix, we have $\1_{A_T(a)} O_{T^{-}}(a)=\1_{A_T(a)}Y_{T^{-}}(a)$ and
$$
(O_{T^{-}}(a)- O_{T}(a)) \1_{A_T(a)} \geq (Y_{T^{-}}(a)- Y_{T}(a)) \1_{A_T(a)} >0.
$$
Therefore $\P[\{T=\infty\}\cap A_T(a)]=0$ since $O_{\infty}(a)=\lim_{t\rw \infty}O_{t}(a)=0.$ 
On the other hand,
\begin{eqnarray*}
\1_{A_T(a)}\big(O_{T^{-}}(a)- O_{T}(a)\big)&= & \1_{A_T(a)}\big\lbrace \E \big\lbrack \max_{\beta \in U} \{ -e^{-r(T+\Delta)}\psi(\beta)+Y_{(T+\Delta)^{-}}(a+\beta)\}|\Fc_{T}\big\rbrack  \\ & {} &\qq \qq \qq -   \E \big\lbrack \max_{\beta \in U}\{-e^{-r(T+\Delta)}\psi(\beta)+Y_{T+\Delta}(a+\beta)\}|\Fc_{T}\big\rbrack \big\rbrace  \\ 
&\leq &
 \1_{A_T(a)}\E\big\lbrack \max_{\beta \in U} \big\lbrace Y_{(T+\Delta)^{-}}(a+\beta)-Y_{T+\Delta}(a+\beta)\big\rbrace|\Fc_{T}\big\rbrack  \\   &\leq &
 \1_{A_T(a)} \E \big\lbrack \max_{\beta \in U} \big\lbrace \1_{A_{T+\Delta}(a+\beta)}\big\lbrace Y_{(T+\Delta)^{-}}(a+\beta)-Y_{T+\Delta}(a+\beta)\big\rbrace |\Fc_{T} \big\rbrack.
\end{eqnarray*}
Since $A_T(a) \in \Fc_{T}$, we get:
\begin{eqnarray}\label{cont1} \nonumber
0<\1_{A_{T}(a)}\{O_{T^{-}}(a)- O_{T}(a)\} \leq 
  \E \big\lbrack \1_{A_{T}(a)}\max_{\beta \in U} \big\lbrace \1_{A_{T+\Delta}(a+\beta)} \big\lbrace Y_{(T+\Delta)^{-}}(a+\beta) -Y_{T+\Delta}(a+\beta) \big\rbrace\big\rbrace  |\Fc_{T} \big\rbrack .
\end{eqnarray}
\def \bbeta{\bar \beta}
It implies that the right-hand side is non negative, which means that there exists a  random variable $\bar \beta$ such that $\bar \beta(\Omega) \subset U$ and $\1_{A_{T+\Delta}(a+\bar \beta)} \big\lbrace Y_{(T+\Delta)^{-}}(a+\bar \beta) -Y_{T+\Delta}(a+\bar \beta) >0$, $\P$-a.s. 
Now, since  $Y_{T+\Delta}(a+\bbeta)\geq O_{T+\Delta}(a+\bbeta)$ and $\1_{A_{T+\Delta}(a+\bbeta)} Y_{(T+\Delta)^{-}}(a+\bbeta)=\1_{A_{T+\Delta}(a+\bbeta)} O_{(T+\Delta)^{-}}(a+\bbeta)$, then
\begin{eqnarray*} 
\1_{A_{T}(a)}\{O_{T^{-}}(a)- O_{T}(a)\} \leq 
  \E \big\lbrack \1_{A_{T}(a)}\max_{\beta \in U} \big\lbrace \1_{A_{T+\Delta}(a+\beta)} \big\lbrace O_{(T+\Delta)^{-}}(a+\beta) -O_{T+\Delta}(a+\beta) \big\rbrace |\Fc_{T} \big\rbrack.
\end{eqnarray*}
So let $\beta_0=0$ and $\beta_1$ the $\Fc_{T+\Delta}$-measurable random variable valued in $U$ (which exists since $U$ is finite) such that 
$$
\max_{\beta \in U} \big\lbrace \1_{A_{T+\Delta}(a+\beta)} \big\lbrace O_{(T+\Delta)^{-}}(a+\beta) -O_{T+\Delta}(a+\beta) \big\rbrace \big\rbrace=\1_{A_{T+\Delta}(a+\beta_{1})} \big\lbrace O_{(T+\Delta)^{-}}(a+\beta_{1}) -O_{T+\Delta}(a+\beta_{1})\}.
$$
It follows that
\begin{eqnarray*} 
\1_{A_{T}(a)}\{O_{T^{-}}(a)- O_{T}(a)\} &\leq& \E \big\lbrack \1_{A_{T}(a)}  \E \big\lbrack \1_{A_{T+\Delta}(a+\beta_{1})} \big\lbrace O_{(T+\Delta)^{-}}(a+\beta_{1}) -O_{T+\Delta}(a+\beta_{1}) \big\rbrace|\Fc_{T+\Delta} \big\rbrack |\Fc_{T} \big\rbrack\\ &\leq & 
 \E \big\lbrack \1_{A_{T}(a)}\times \1_{A_{T+\Delta}(a+\beta_{1})} \big\lbrace O_{(T+\Delta)^{-}}(a+\beta_{1}) -O_{T+\Delta}(a+\beta_{1}) \big\rbrace |\Fc_{T} \big\rbrack .
\end{eqnarray*}
Repeating this reasoning as many times as necessary to obtain random variables $\beta_{k}$  valued in $U$ and $\Fc_{T+k\Delta}$-measurable, $k=1,...,n$, such that: 
\begin{eqnarray*}\label{cont3}
0<\1_{A_{T}(a)}\{O_{T^{-}}(a)& -& O_{T}(a)\} \leq 
 \E \big\lbrack \big\lbrace \prod_{k=0}^{n} \1_{A_{T+k\Delta}(a+\beta_1+ \cdots +\beta_k)} \\ &\times &\big( O_{(T+n\Delta)^{-}}(a+\beta_1+ \cdots +\beta_n)-O_{T+n\Delta}(a+\beta_1+ \cdots +\beta_n)\big) \big\rbrace |\Fc_{T}\big\rbrack\\
 &{}&\qq\le 
 2\E[Ce^{-(n+1)\d-T}+\int_{T+n\d}^\infty e^{-rs}{\gamma}_{s}ds|\Fc_{T}]\rw 0 \mbox{ as $n \rw \infty$}.
\end{eqnarray*}
The last inequality stems from the definition of the process $(O_t(a))_{t\in [0,\infty[}$ in \eqref{infO} and then for and $a\in \rl$, for any $t\ge 0$ we have, 
$$
|O_t(a)|\le \E[\int_t^\infty e^{-rs}\g_sds+Ce^{-t-\d}|\Fc_t]
$$
by using Assumption \ref{assumpt3}-i), i) of Proposition \ref{limyn} and the finitness of $U$. Therefore, 
$$
0<\1_{A_{T}(a)}\{O_{T^{-}}(a)-O_{T}(a)\}=0,
$$
which is absurd. Thus for any $a\in \rl$, the process $(Y_t(a))_{t\in [0,\infty]}$ is continuous. Finally the process $(O_t(a))_{t\in [0,\infty]}$ is then so. 
\end{proof}
\begin{Theorem}\label{infoptimal}  There exists an optimal strategy $\delta^{*,\infty} = (\tau_{n}^{*},\beta_{n}^{*})_{n\geq 0}$ for the infinite horizon stochastic impulse control problem, i.e., 
 $$ Y_{0}(0)=\Sup_{\delta \in \mathcal{A}}J(\delta)=J(\delta^{*,\infty}).$$ 
\end{Theorem}
\begin{proof} We first exhibit the strategy $\delta^{*,\infty}$. Let
$$\tau_{1}^{*}= \inf\{ s\in[0,+\infty];O_{s}(0)= Y_{s}(0)\}$$ 
and $\beta_{1}^{*}$ an $U$-valued random variable, $\Fc_{\tau_{1}^{*} + \Delta}-$measurable such that
$$ O_{\tau_{1}^{*}}(0):= \E\left[\Int_{\tau_{1}^{*}}^{\tau_{1}^{*}+\Delta}  e^{-rs}  g(s,X_s)ds- e^{-r(\tau_{1}^{*}+\Delta)}\psi(\beta_{1}^{*})+Y_{\tau_{1}^{*}+\Delta}(\beta_{1}^{*})|\mathcal{F}_{\tau_{1}^{*}}\right].$$
 For any $n \geq 2$ 
 $$\tau_{n}^{*}=\inf\bigg\lbrace s\in [\tau_{n-1}^{*}+\Delta,+\infty], O_{s}(\beta_{1}^{*}+\cdots +\beta_{n-1}^{*} )= Y_{s}(\beta_{1}^{*}+\cdots +\beta_{n-1}^{*} )\bigg\rbrace$$
 and $\beta_n^*$ an $\Fc_{\tau_n^* + \Delta}-$measurable, $U$-valued random variable satisfying
 \begin{eqnarray*}
 O_{\tau_{n}^{*}}(\beta_{1}^{*}+\cdots +\beta_{n-1}^{*} )&=&\E \bigg\lbrack \Int_{\tau_{n}^{*}}^{\tau_{n}^{*}+\Delta}  e^{-rs}  g(s,X_{s}+\beta_{1}^{*}+\cdots +\beta_{n-1}^{*} )ds \\ &-& e^{-r(\tau_{n}^{*}+\Delta)}\psi(\beta_{n}^{*})+Y_{\tau_{n}^{*}+\Delta}(\beta_{1}^{*}+\cdots +\beta_{n}^{*} )|\mathcal{F}_{\tau_{n}^{*}}\bigg\rbrack.
 \end{eqnarray*}
 Let us highlight first that the stopping times $\t_n^*$, $n\ge 1$, are optimal. Indeed, for any $a\in \rl$, $(Y_t(a))_{t\in [0,\infty]}$ and $(O_t(a))_{t\in [0,\infty]}$ are continuous. Moreover,   taking \eqref{eqlim} into account    and using the fact that $Y_\infty(a)=\lim_{t\rw \infty}O_t(a)=0$, the optimality follows.
 
We now divide the proof into two steps:

\nd \textbf{Step 1:} First let us show that $Y_{0}(0)= J(\delta^{*,\infty})$. Note that $0$ in the index is related to the initial time while $0$ in the parentheses is related to the initial state where an impulse is not made yet.
\ms

\nd Since $\t_1^*$ is optimal and $\b_1^*$ is the optimal impulse then  
\begin{eqnarray*}
 Y_{0}(0)&=& \esssup_{\tau \in
\mathcal{T}_0}\E\left[\int_{0}^{\tau} e^{-rs}g(s,X_s)
ds+  O_{\tau}(0)\right]\\
&=& \E\left[\int_{0}^{\tau_{1}^{*}} e^{-rs}g(s,X_s)
ds+  O_{\tau_{1}^{*}}(0)\right]\\
&=& \E\left[\int_{0}^{\tau_{1}^{*}+\Delta} e^{-rs}g(s,X_s) ds -  e^{-r(\tau_{1}^{*}+\Delta)}\psi(\beta_{1}^{*})+  Y_{\tau_{1}^{*}+\Delta}(\beta_{1}^{*})\right].
\end{eqnarray*}
On the other hand, we have 
\begin{eqnarray*}
Y_{\tau_{1}^{*}+\Delta}(\beta_{1}^{*})& = &  \esssup_{\tau \in
\mathcal{T}_{\tau_{1}^{*}+\Delta}}\E\left[\int_{\tau_{1}^{*}+\Delta}^{\tau} e^{-rs}g(s,X_{s}+\beta_{1}^{*})
ds+  O_{\tau}(\beta_{1}^{*})|\mathcal{F}_{\tau_{1}^{*}+\Delta}\right]\\&=&
\E\left[\int_{\tau_{1}^{*}+\Delta}^{\tau_{2}^{*}} e^{-rs}g(s,X_{s}+\beta_{1}^{*})
ds+  O_{\tau_{2}^{*}}(\beta_{1}^{*})|\mathcal{F}_{\tau_{1}^{*}+\Delta}\right] \\
& =& \E\left[\int_{\tau_{1}^{*}+\Delta}^{\tau_{2}^{*}+\Delta} e^{-rs}g(s,X_{s}+\beta_{1}^{*}) ds -  e^{-r(\tau_{2}^{*}+\Delta)}\psi(\beta_{2}^{*})+  Y_{\tau_{2}^{*}+\Delta}(\beta_{1}^{*}+\beta_{2}^{*})|\mathcal{F}_{\tau_{1}^{*}+\Delta}\right].
\end{eqnarray*} 
Therefore,
\begin{eqnarray*}
 Y_{0}(0) &=& \E \bigg\lbrack \int_{0}^{\tau_{1}^{*}+\Delta} e^{-rs}g(s,X_s) ds + \int_{\tau_{1}^{*}+\Delta}^{\tau_{2}^{*}+\Delta} e^{-rs}g(s,X_{s}+\beta_{1}^{*}) ds \\ &{}&\qq\qq   - e^{-r(\tau_{1}^{*}+\Delta)}\psi(\beta_{1}^{*})-e^{-r(\tau_{2}^{*}+\Delta)}\psi(\beta_{2}^{*})+  Y_{\tau_{2}^{*}+\Delta}(\beta_{1}^{*}+\beta_{2}^{*})\bigg\rbrack.
\end{eqnarray*}
Repeating $n$ times this reasoning to obtain:
\begin{eqnarray*}
 Y_{0}(0)  &= &\E \bigg\lbrack \int_{0}^{\tau_{1}^{*}+\Delta} e^{-rs}g(s,X_s) ds  + \sum_{1\leq k \leq n-1} \int_{\tau_{k}^{*}+\Delta}^{\tau_{k+1}^{*}+\Delta} e^{-rs}g(s,X_{s}+\beta_{1}^{*}+\cdots+\beta_{k}^{*}) ds \\ \qquad &{}&\qq\qq -  \sum_{k=1}^{n} e^{-r(\tau_{k}^{*}+\Delta)}\psi(\beta_{k}^{*})+Y_{\tau_{n}^{*}+\Delta}(\beta_{1}^{*}+\cdots + \beta_{n}^{*})\bigg\rbrack.
\end{eqnarray*}
It implies that 
\begin{align*}
|Y_0(0)-J(\dl^{*,\infty})|&\le 
 \E \bigg\lbrack  \sum_{k\geq n} \int_{\tau_{k}^{*}+\Delta}^{\tau_{k+1}^{*}+\Delta} e^{-rs}|g(s,L_{s}+\beta_{1}^{*}+\cdots+\beta_{k}^{*})|ds +  \\&\qq\qq\qq\qq \sum_{k\ge n+1}e^{-r(\tau_{k}^{*}+\Delta)}\psi(\beta_{k}^{*})+|Y_{\tau_{n}^{*}+\Delta}(\beta_{1}^{*}+\cdots + \beta_{n}^{*})|\bigg\rbrack\\
 &\le 2\E[\int_{n\d}^\infty e^{-rs}\g_sds+Ce^{-n\d}]
\end{align*}
where $C$ is a constant which does not depend on $n$. This latter inequality stems from Assumption \ref{assumpt3}-i), the inequality of Proposition \ref{limyn}-i) and finally the delay $\d$ between two successif stopping times. Finally take the limit w.r.t $n$ to obtain that 
$Y_0(0)=J(\dl^{*,\infty})$. 
\vskip0.2cm 

\nd \textbf{Step 2:}  $Y_0(0)\geq J(\delta ^{\prime})$ for any $\delta^{\prime} \in \mathcal{A}_\infty$. \\
By virtue of Proposition \ref{limyn}-ii), we have:
 $$ Y_{0}(0)\geq \E\left[\int_{0}^{\tau_{1}^{\prime}} e^{-rs}g(s,X_s)
ds+  O_{\tau_{1}^{\prime}}(0)\right].$$ 
However,
 $$ O_{\tau_{1}^{\prime}}(0) \geq \E\left[\Int_{\tau_{1}^{\prime}}^{\tau_{1}^{\prime}+\Delta}  e^{-rs}  g(s,X_s)ds- e^{-r(\tau_{1}^{\prime}+\Delta)}\psi(\beta_{1}^{\prime})+  Y_{\tau_{1}^{*}+\Delta}(\beta_{1}^{\prime})|\mathcal{F}_{\tau_{1}^{\prime}}\right].$$
Combining the last inequality with the previous one, we obtain:  
\begin{eqnarray}
 Y_{0}(0)\geq \E\left[\int_{0}^{\tau_{1}^{\prime}+\Delta} e^{-rs}g(s,X_s)ds - e^{-r(\tau_{1}^{\prime}+\Delta)}\psi(\beta_{1}^{\prime})+  Y_{\tau_{1}^{\prime}+\Delta}(\beta_{1}^{\prime})\right].
\end{eqnarray}
On the other hand, we have that
\begin{eqnarray*}
Y_{\tau_{1}^{'}+\Delta}(\beta_{1}^{\prime})& =& \esssup_{\tau \in \mathcal{T}_{\tau_{1}^{\prime}+\Delta}}\E \left[
\int_{\tau_{1}^{\prime}+\Delta}^{\tau} e^{-rs}g(s,X_{s}+\beta_{1}^{\prime})ds + O_{\tau}(\beta_{1}^{\prime})|\Fc_{\tau_{1}^{\prime}+\Delta}\right]\\
& \geq & \E \left[ \int_{\tau_{1}^{\prime}+\Delta}^{\tau_{2}^{\prime}} e^{-rs}g(s,X_{s}+\beta_{1}^{\prime})ds+ O_{\tau_{2}^{\prime}}(\beta_{1}^{\prime})|\Fc_{\tau_{1}^{\prime}+\Delta}\right]\\
&\geq &\E \left[ \int_{\tau_{1}^{\prime}+\Delta}^{\tau_{2}^{\prime}+\Delta} e^{-rs}g(s,X_{s}+\beta_{1}^{\prime})ds-  e^{-r(\tau_{2}^{\prime}+\Delta)}\psi(\beta_{2}^{\prime})+  Y_{\tau_{2}^{\prime}+\Delta}(\beta_{1}^{\prime}+\beta_{2}^{\prime})|\Fc_{\tau_{1}^{\prime}+\Delta}\right].
\end{eqnarray*} 
Therefore,
\begin{eqnarray*}
Y_{0}(0) & \geq & \E \bigg\lbrack \int_{0}^{\tau_{1}^{\prime}+\Delta} e^{-rs}g(s,X_s)ds+ \int_{\tau_{1}^{\prime}+\Delta}^{\tau_{2}^{\prime}+\Delta} e^{-rs}g(s,X_{s}+\beta_{1}^{\prime})ds \\ &-&  e^{-r(\tau_{1}^{\prime}+\Delta)}\psi(\beta_{1}^{\prime})- e^{-r(\tau_{2}^{\prime}+\Delta)}\psi(\beta_{2}^{\prime})+  Y_{\tau_{2}^{\prime}+\Delta}(\beta_{1}^{\prime}+\beta_{2}^{\prime})\bigg\rbrack.\end{eqnarray*}
In repeating this reasoning as many times as necessary, we obtain:
\begin{eqnarray*}
Y_{0}(0) & \geq & \E \bigg\lbrack \int_{0}^{\tau_{1}^{\prime}+\Delta} e^{-rs}g(s,X_s)ds+\sum_{1\leq k \leq n-1}\int_{\tau_{k}^{\prime}+\Delta}^{\tau_{k+1}^{\prime}+\Delta} e^{-rs}g(s,X_{s}+\beta_{1}^{\prime}+\cdots + \beta_{k}^{\prime})ds \\ &-& \sum_{k=1}^{n}e^{-r(\tau_{k}^{\prime}+\Delta)}\psi(\beta_{k}^{\prime})+ Y_{\tau_{n}^{\prime}+\Delta}(\beta_{1}^{\prime}+\cdots +\beta_{n}^{\prime})\bigg\rbrack.
\end{eqnarray*}
Finally, by taking the limit as $n\to +\infty$, we get 
$$ Y_{0}(0) \geq  \E \left[\int_{0}^{+\infty } e^{-rs}g(s,X_{s}^{\delta^{\prime}})ds - \sum_{n\geq 1}e^{-r(\tau_{n}^{\prime}+\Delta)}\psi(\beta_{n}^{\prime})\right]= J(\delta^{\prime})$$
since by Proposition \ref{limyn}-i), \eqref{eq4x}, \eqref{limm} and \eqref{eqpi}, $\lim_{n\rw \infty}
Y_{\tau_{n}^{\prime}+\Delta}(\beta_{1}^{\prime}+\cdots +\beta_{n}^{\prime})=0$ $\P$-a.s. and in $L^1(d\P)$. Hence, the strategy $\delta^{*,\infty}$ is optimal
 \end{proof}
\begin{Remark}\lb{rmkrs}It is also possible to consider the infinite horizon risk-sensitive case. As this setting does not rise major difficulties w.r.t the risk neutral framework we considered, then it is left to the care of the reader. Only small adaptations are needed especially the integrability of the function $g(.)$ and its majorizing process $(\g_t)_{t\ge 0}$ (see Assumption \ref{assumpt-expo}). 
\end{Remark}

\section{Appendix}
\subsection{Optional and predictable projections}\label{section-projection}
 The acronymes $rcll$ and $lcrl$ refer respectively to right continuous left limited and left continuous right limited.
Next let  $(\Omega,\mathcal{F},\P)$
be a complete probability space which carries a filtration $(\mathcal{F}_{t})_{t\geq 0}$ which satifies the usual conditions. 
 
\begin{Theorem} (\cite{Dellacherie}, pp. 116) 
Let $X:=(X_t)_{t\geq 0}$ be a measurable process verifying \\$\E[\sup_{t\ge 0}|X_t|]<\infty$. Then there exists a unique (up to indistinguishability)  optional process $^{o}X$ such that for any $(\mathcal{F}_{t})_{t\geq 0}$-stopping time $\tau$,
$$\E( X_{\tau}\1_{\{\tau<\infty\}}|\mathcal{F}_{\tau})=\, ^{o}X_{\tau}\1_{\{\tau<\infty\}}.$$
Likewise, there exist a unique predictable projection $^{p}X$ such that: for any $(\mathcal{F}_{t})_{t\geq 0}$-predictable stopping time $\tau$,
$$\E( X_{\tau}\1_{\{\tau<\infty\}}|\mathcal{F}_{\tau-})=\, ^{p} X_{\tau}\1_{\{\tau<\infty\}}.$$
The process $^{o}X$ (resp. $^{p}X$) is called the optional (resp. predictable) projection of $X$.
\end{Theorem}
\begin{Theorem}\lb{thmcontproj}
(\cite{Dellacherie}, pp.119, 125)
Let $X$ be a measurable  process such that 
$\E[\sup_{t\ge 0}|X_t|]<\infty$ and $^{o}X$, $^{p}X$  are respectively its optional  and predictable projections.
\begin{enumerate}[i)]
\item If $X$ is right continuous (resp. rcll), $^{o}X$ is right continuous (resp. rcll).
\item If $X$ is left continuous (resp. lcrl), $^{p}X$ is left continuous(resp. lcrl).
\end{enumerate}
\label{della}
\end{Theorem}
\begin{Lemma}\label{cont}
For the Brownian filtration, the optional and predictable $\sigma-$fields coincide which means that the optional projection of a continuous process $X$ such that $\E[\sup_{t\ge 0}|X_t|]<\infty$ is also continuous.
\end{Lemma}
For more details, we refer to  \cite{Revuz}, Theorem 3.5, pp.201 and Corollary 5.7, pp.174.
\begin{Remark}Let $T>0$ be a real number. The predictable and optional projections of a stochastic measurable process $X=(X_t)_{t\in [0,T]}$ are the ones of the process $X^T=(X_{t\wedge T})_{t\ge 0}$ on the filtered probability space $(\Omega,\mathcal{F},\P,
(\mathcal{F}_{t\wedge T})_{t\geq 0})$.
\end{Remark}
\subsection{Snell envelope of a process}
Let $T\in [0,+\infty]$. An $(\mathcal{F}_{t})_{t\leq T}$-adapted rcll process $U=(U_{t})_{0\leq t\leq T}$, is called of class $[D]$, if the family $\{U_{\tau}\}_{\tau\in \mathcal{T}_{0}}$ is uniformly integrable.
\begin{Proposition}\label{snellenv}(see \cite{Dellacherie}, pp.432) 
Let $U=(U_{t})_{0\leq t\leq T}$ be a process of class $[D]$, then there exists an
$(\mathcal{F}_{t})_{t\leq T}$-adapted $\R$-valued $rcll$ process $N:=(N_{t})_{t\leq T}$ such that:
\begin{enumerate}[i)]
\item $N$ is the smallest $\cadlag$ supermartingale of class [D], called the Snell envelope of $U$, which dominates $U$, i.e., if $(N'_{t})_{0\leq t\leq T}$ is another
c\`{a}dl\`{a}g supermartingale of class $[D]$ such that, $N_{t}'\geq U_{t}$ then $N_{t}'\geq N_{t}$ for any $0\leq t\leq T$.
\item For any $t\leq T$, it holds: 
\begin{equation}\label{eq 1}
 N_{t}=\esssup_{\tau \in
\mathcal{T}_t}\E[U_{\tau}|\mathcal{F}_{t}] \,\, (N_{T}=U_{T}).
\end{equation}
\item    The jumping times of $(N_{t})_{t\leq T}$ are predictable and satisfy:
$$\{\Delta_tN<0\}\subset \{N_{t^{-}}=U_{t^{-}}\}\cap \{\Delta_{t} U<0\}.$$
\end{enumerate}
\end{Proposition}
\begin{Proposition}\label{increaseconv}(\cite{Djehall}, Proposition 2)
    Let $(U_n)_{n\geq 0}$ and $U$ be c\`adl\`ag processes of class $[D]$  such that the sequence of processes
$(U_n)_{n\geq 0}$ converges increasingly and pointwisely to $U$, then the snell envelope of $ (U_{n})_{n\geq 0}$ converges increasingly and pointwisely to the one of $U$. 
\end{Proposition}
\begin{Proposition}\label{cont-opt} (\cite{El-Karoui1},  Proposition II.2.7)
If $U$ is right upper semi-continuous, then the process $N$ is continuous. Furthermore, if $\gamma$ is an $(\mathcal{F}_{t})_{t\leq T}$-stopping time and \\ $\tau_{\gamma}^{\ast}=\inf\{s\geq\gamma, N_{s}=U_{s}\}\wedge T$, then $\tau_{\gamma}^{\ast}$ is optimal after $\gamma$, i.e.,
\begin{equation}
N_{\gamma}=\esssup_{\tau \in
\mathcal{T}_\gamma}\E[U_{\tau}|\mathcal{F}_{\gamma}]=\E[U_{\tau_{\gamma}^{\ast}}|\mathcal{F}_{\gamma}]=\E[N_{\tau_{\gamma}^{\ast}}|\mathcal{F}_{\gamma}].
\label{optimtime}
\end{equation}
\end{Proposition}
\end{document}